\documentclass[12pt]{article}
\usepackage{graphicx}
\usepackage{picture}
\usepackage{color}
\usepackage[english]{babel}
\usepackage{amsmath}
\usepackage{amssymb}
\usepackage{enumerate}
\usepackage{amsfonts}
\usepackage{layout}

\usepackage{pgf,tikz}
\usetikzlibrary{arrows}
\usepackage{caption}
\usepackage{hyperref}
\hypersetup{colorlinks=false, linkbordercolor={1 1 1}, citebordercolor={1 1 1}, urlbordercolor={1 1 1}} 

\usepackage[hmargin=25mm,top=20mm,bottom=25mm]{geometry}     
\usepackage{latexsym}
\usepackage{amsthm}
\newtheorem{theorem}{Theorem}
\newtheorem{lemma}[theorem]{Lemma}
\newtheorem{corollary}[theorem]{Corollary}

\newtheorem{proposition}[theorem]{Proposition}

\newcommand{\nn}{\nonumber}
\newcommand{\dis}{\displaystyle}

\newcommand{\ux}{X}
\newcommand{\hux}{\hat X}
\newcommand{\huy}{\hat Y}
\newcommand{\hY}{\hat Y}

\newcommand{\uy}{Y}
\newcommand{\uz}{Z}
\newcommand{\usi}{\sigma}

\newcommand{\ro}{\rho}

\newcommand{\R}{\mathbb R}

\newcommand{\N}{\mathbb N}

\newcommand{\cL}{\mathcal{L}}

\newcommand{\cB}{\mathcal{B}}

\newcommand{\one}{{\mathbf 1}}

\let\originalleft\left
\let\originalright\right
\renewcommand{\left}{\mathopen{}\mathclose\bgroup\originalleft}
\renewcommand{\right}{\aftergroup\egroup\originalright}
\newcommand{\llav}[1]{  \left\{#1\right\} }

\newcommand{\norm}[1]{  \left\|#1\right\| }
\newcommand{\pare}[1]{\left(#1\right)}
\newcommand{\corch}[1]{  \left[#1\right] }
\newcommand{\abs}[1]{  \left|#1\right| }
\renewcommand{\upsilon}{\psi}
\def\defi{\mathrel{\mathop:}=}

\numberwithin{equation}{section}

\def\cle{\preccurlyeq}
\def\cge{\succcurlyeq}

\begin{document}

\title{Hydrodynamics of the $N$-BBM process}



\author{Anna De Masi \and Pablo A. Ferrari \and Errico Presutti \and Nahuel Soprano-Loto}



\date{}

\maketitle
\parindent 0pt

\paragraph{Abstract} The Branching Brownian Motions (BBM) are particles performing independent Brownian motions in $\R$ and each particle at rate 1 creates a new particle at her current position; the newborn particle increments and branchings are independent of the other particles. The $N$-BBM starts with $N$ particles and at each branching time, the leftmost particle is removed so that the total number of particles is $N$ for all times. The $N$-BBM was proposed by Maillard and belongs to a family of processes introduced by Brunet and Derrida. We fix a density $\rho$ with a left boundary $L=\sup\{r\in\mathbb R: \int_r^\infty \rho(x)dx=1\}>-\infty$ and let the initial particle positions be iid continuous random variables with density $\rho$. We show that the empirical measure associated to the particle positions at a fixed time $t$ converges to an absolutely continuous measure with density $\psi(\cdot,t)$, as $N\to\infty$. The limit $\psi$ is solution of a free boundary problem (FBP) when this solution exists. The existence of solutions for finite time-intervals has been recently proved by Lee.   

\paragraph{Keywords} Hydrodynamic limit. Free boundary problems. Branching Brownian motion. Brunet-Derrida systems.

%
%
%
%
%
%
%
%
%
%
%
%
%


\section{Introduction}
\label{s1}
This is a version of a previous work by the same authors based on a talk
by the second author at the Institut Henri Poincar\'e in  June 2017.

Brunet and Derrida \cite{brunet-derrida} proposed a family of one dimensional processes with $N$ branching particles with selection.
Start with $N$ particles with positions in $\R$.
At each discrete time $t$, there are two steps.
In the first step, each particle creates a number of descendants at positions chosen according to some density as follows: if a particle is located at position $x$, then its descendants are iid with distribution $Y+x$, where $Y$ is a random variable with a given density.
The second step is to keep the $N$ right-most particles, erasing the left-most remaining ones. 

The study of $N$-Branching Brownian motions was proposed by Pascal Maillard \cite{MR3568046,MR3088376} as a natural continuous time version of the previous process, also related with the celebrated Branching Brownian motion process, that we abbreviate by BBM.
The $N$-BBM move as independent Brownian motions, and each particle at rate one creates a new particle at its current position. When a new particle is created, the leftmost particle is removed. The number $N$ of particles is then conserved.

The particles are initially distributed as independent random variables with an absolutely continuous distribution whose density is called $\ro$.
Let $\ux_t=\{\ux^1_t,\dots,\ux^N_t\}$ be the set of positions of the $N$ particles at time $t$. Here and in the sequel we will consider multi-sets, allowing repetitions of elements. Denote $|A|$ the cardinal of a discrete set $A$. 
The empirical distribution induced by $\ux_t$ is defined by
	 \begin{equation}
\label{pit}
\pi_t^{N}[a,\infty):= \frac 1N \abs{\ux_t\cap [a,\infty)},
   \end{equation}
the proportion of particles to the right of $a$ at time $t$. Our main result is the following hydrodynamic limit.

\begin{theorem}
\label{hydrodynamics}
Let $\ro\in L^1(\R,\R^+)$ be a probability density function satisfying $\|\rho\|_\infty<\infty$ and $L_0:=\sup_{r}\big\{\int_{-\infty}^r\ro(x)dx=0\big\}>-\infty$.
Let $\ux^1_0,\dots,\ux^N_0$ be independent identically distributed continuous random variables with density $\ro$. Let $\ux_t$ be the positions at time $t$ of  $N$-\nobreak{\rm BBM} starting at $\ux_0=\{\ux^1_0,\dots,\ux^N_0\}$.
For every $t\ge 0$, there is a density function $\psi(\cdot,t):\R\to\R^+$ such that, for any $a\in \R$, we have
 	\begin{equation*}
\lim_{N\to \infty} \int_a^\infty \pi_t^{N}(dr) =\int_a^\infty \psi(r,t)dr,\quad \hbox{a.s.\/ and in $L^1$.}
   \end{equation*}
	\end{theorem} 
In Theorem \ref{u-existence} below, we identify  the function $\psi(r,t)$ as  the local solution $u(r,t)$ of the following free boundary problem.

\paragraph{Free boundary problem (FBP).} Find $(u,L)\equiv((u(\cdot,t), L_t):t \in [0,T])$ such that:
\begin{eqnarray}
&&u_t= \frac 12  u_{rr} +u,\qquad  {\rm in}\;\; (L_t,+\infty);\label{u2}\\
&&u(r,0)= \ro(r);\qquad  \ro(r)=0 \text{ for }r\le L_0, \int_{L_0}^\infty \ro(r)dr=1\label{u3}\\
&&u(L_t,t)=0,\qquad\int_{L_t}^\infty u(r,t)dr=1. \label{u4}
\end{eqnarray}
Observe that if $(u,L)$ is a solution such that the right derivative $u_r(L_t,t)$ exists, then by \eqref{u4} we have $\frac 12 u_r(L_t,t) =\int_{L_t}^\infty u(r,t)dr$. Furthermore, if the second right derivative exists, then $\dot L_t = -\frac12 u_{rr}(L_t, t)$. Berestycki, Brunet and Derrida \cite{bbd} propose a family of free boundary problems which include this one and give an explicit relation between $\rho$ and $L$, under certain conditions. 


Lee \cite{JLee} proved that if $\ro\in C^2_c([L_0,\infty))$ and $\rho'_{L_0}=2$ then there exist $T>0$ and a solution $(u,L)$ of the  free boundary problem \eqref{u2}-\eqref{u4} with the following properties. The curve $\{L_t: t\in[0,T]\}$ is in $C^1$, $L_{t=0}=L_0$  and the function $u\in C(\overline{D_{L,T}})\cap C^{2,1}(D_{L,T})$ where $D_{L,T}=\{(r,t):L_t<r, 0<t<T\}$.
\begin{theorem}
\label{u-existence}
Let $((u(\cdot,t),L_t):t\in[0,T])$ be a solution of the free boundary problem \eqref{u2}-\eqref{u4} with $L$ continuous. Then the function $\psi$ defined in Theorem \ref{hydrodynamics} coincides with $u$ in that time interval: $\psi(\cdot,t) = u(\cdot,t), t\in[0,T]$. In particular, $(u,L)$ is the unique solution in that time interval. 
\end{theorem}
Theorem \ref{u-existence} together with Lee's result imply that under Lee conditions on $\rho$, the empirical measure of the $N$-BBM starting with iid with density $\rho$ converges to the solution of FBP in the sense of Theorem \ref{hydrodynamics} in the time interval $[0,T]$. To obtain the convergence in any time interval and for an arbitrary density it suffices to show that there exists a solution $(u,L)$ for $t\in\R^+$ with $L$ continuous. However this is an open problem, see \cite{bbd}.

We observe that if  $(u,L)$ is a solution for the FBP and $L$ is continuous, then we have the following Brownian motion representations of the solutions. 
\begin{align}
\label{urtbm}
  \int_a^\infty u(r,t) dr = e^t\int \rho(x) P_x(B_t>a, \tau^L>t)dx,  
\end{align}
where $P_x$ is the law of Brownian motion with initial position $B_0=x$ and $\tau^L:= \inf\{s: B_s\le L_s\}$ is the hitting time of $L$. Taking $a=-\infty$ in \eqref{urtbm} we get
\begin{align}
\label{emt}
   e^t\int \rho(x) P_x(\tau^L>t)dx =1.
\end{align}
which implies that $\tau^L$ is an exponential random variable of mean 1. In other words, we are looking for a continuous curve $L$ such that if Brownian motion starts with a random initial position with density $\rho$, the hitting time of $L$ is exponentially distributed.

 The density solution has a backwards representation:
\begin{align}
\label{urtbm2}
  u(x,t) = e^tE_x \big[\rho(B_t) \one\{B_s> L_{t-s}: 0\le s\le t\}\big].
\end{align}
The presence of the free boundary at the left-most particle spoils usual hydrodynamic proofs.
We overcome the difficulty by dominating the process from below and above by auxiliary more tractable processes, a kind of Trotter-Kato approximation. Durrett and Remenik \cite{MR2932664} use an upperbound to show the analogous to Theorem \ref{hydrodynamics} for a continuous time Brunet-Derrida model. The approach with upper and lower bounds is used by three of the authors in \cite{MR3304749} and by Carinci, De Masi, Giardina and Presutti  \cite{MR3198665}, \cite{MR3282863}, see the survey \cite{MR3497333}; a further example is \cite{MR3336872}.  Maillard \cite{MR3568046} used upper and lower bounds with a different scaling and scope. In Durrett and Remenik the leftmost particle motion is increasing and has natural lower bounds. The lower bounds used in the mentioned papers do not work out-of-the-box here. We  introduce labelled versions of the processes and a trajectory-wise coupling to prove the lowerbound in Proposition \ref{domination} later. 

In Section \ref{s2} we introduce the elements of the proof of hydrodynamics, based on approximating barriers that will dominate the solution from above and below. In Section \ref{s3} we construct the coupling to show the dominations. In Section \ref{s4} we show the hydrodynamics for the barriers. In Section \ref{s5} is devoted to the proof of the existence of the limiting density $\psi$. Section \ref{s6} we prove Theorem \ref{hydrodynamics}. Section \ref{s7} proves Theorem \ref{u-existence}. Finally in Section \ref{s8} we state a Theorem for fixed $N$ establishing the existence of a unique invariant measure for the process as seen from the leftmost particle and a description of the traveling wave solutions for the FBP.

\section{Domination and barriers}
\label{s2}

We define the $N$-BBM process and the limiting barriers as functions of a ranked version of the BBM process.
\paragraph{Ranked BBM.}
Denote by $(B^{1,1}_0, \dots, B^{N,1}_0)$ the initial positions of the BBM starting with $N$ particles. 
Let $N_t^i$ be the size of the $i$-th  BBM family (starting at $B_0^{i,1}$).
For $1\le j\le N_t^i$, let $B_t^{i,j}$ be the position of the $j$-th member of the $i$th family,
 ordered by birth time. Call $(i,j)$ the \emph{rank} of this particle and 
denote
\begin{align}
\label{bijt}
  B^{i,j}_{[0,t]} := \hbox{ trajectory of the $j$-th offspring with initial particle $i$ in the interval $[0,t]$},
\end{align}
with the convention that, before its birth time, the trajectory coincides with those of its ancestors. Define the \emph{ranked} BBM as
\begin{align}
\label{cB}
  \cB := \big(B^{i,j}_{[0,\infty)} : i\in\{1,\dots,N\}, j\in\mathbb N\big).
\end{align}
The positions occupied by the particles at time $t$  
\begin{align}
\label{zbt}
  \uz_t(\cB)&=\{ B^{i,j}_t: 1\le i\le N,\,1\le j\le N^i_t\}
\end{align}
is the BBM.
We drop the dependence on $\cB$ when it is clear from the context. 
\paragraph{$N$-BBM as function of the ranked BBM.} Let $\tau_n$ be the branching times of BBM. We define $L_{\tau_n}$ iteratively: let $X_0=Z_0$, $\tau_0=0$ and
\begin{align}
L_{\tau_n}&:= a\in Z_{\tau_n}\hbox{ such that } \sum_{i=1}^N\sum_{j=1}^{N^i_{\tau_n}}\one\{B^{i,j}_{\tau_{n-1}}\in X_{\tau_{n-1}}: B^{i,j}_{\tau_n}\ge a\}=N\nn \\
  X_{\tau_n}&:= \{B^{i,j}_{\tau_n}: B^{i,j}_{\tau_{n-1}}\in X_{\tau_{n-1}}\hbox{ and } B^{i,j}_{\tau_n}\ge L_{\tau_n}\}, \label{xt1} 
\end{align}
with the convention that, if the branching point at time $\tau_n$ is at $L_{\tau_n}$, and $B^{i,j}_{\tau_n}=B^{i,j'}_{\tau_n}$, $j>j'$, are the two offprings at that time, only $B^{i,j}_{\tau_n}\ge L_{\tau_n}$ while we abuse notation by declaring $B^{i,j'}_{\tau_n}< L_{\tau_n}$.
The process 
\begin{align}
  X_t(\cB)&:=\{B^{i,j}_t: B^{i,j}_{\tau_{\ell}}\ge L_{\tau_\ell}\hbox{ for all } \tau_\ell\le t\} 
       \label{xt7}
\end{align}
is a version of the $N$-BBM described in the introduction.
\paragraph{Stochastic barriers.}
For each positive real number $\delta$, we define the \emph{stochastic barriers}.
These are discrete-time processes denoted by $\ux^{\delta,+}_{k\delta}$ and $\ux^{\delta,-}_{k\delta}$, $k\in\N$, with initial configurations $\ux^{\delta,\pm}_0=\uz_0$.
Iteratively, assume $\ux^{\delta,\pm}_{(k-1)\delta}\subset \uz_{(k-1)\delta}$ is defined. The barriers at time $k\delta$ are selected points of $\uz_{k\delta}$ of cardinal at most $N$ that are defined as follows.  

\emph{The upper barrier}. The selected points at time $k\delta$ are the $N$ rightmost offsprings of the families of the selected points at time $(k-1)\delta$.
The cutting point and the corresponding selected set at time $k\delta$ are
\begin{align}
   L^{\delta,+}_{k\delta}&:= a\in \uz_{k\delta} \hbox{ such that } \sum_{i=1}^N\sum_{j=1}^{N^i_{k\delta}} \one\big\{B^{i,j}_{(k-1)\delta}\in \ux^{\delta,+}_{(k-1)\delta},\,B^{i,j}_{k\delta}\ge a\big\} = N \nn\\
\label{2-2}
    \ux^{\delta,+}_{k\delta}&\defi  \big\{B^{i,j}_{k\delta} :B^{i,j}_{(k-1)\delta}\in \ux^{\delta,+}_{(k-1)\delta}\hbox{ and } B^{i,j}_{k\delta}\ge L^{\delta,+}_{k\delta}\big\}.
  \end{align}
The number of particles in $\ux^{\delta,+}_{k\delta}$ is exactly $N$ for all $k$.

\emph{The lower barrier}.
The selection is realized at time $(k-1)\delta$. Cut particles from left to right at time $(k-1)\delta$ until the largest possible number non bigger than $N$ of particles is kept at time $k\delta$. The cutting point at time $(k-1)\delta$ and the resulting set at time $k\delta$ are given by
\begin{align}
 L^{\delta,-}_{(k-1)\delta}&\defi \min\Bigl\{ a\in  \ux^{\delta,-}_{(k-1)\delta}: \sum_{i=1}^N\sum_{j=1}^{N^i_{k\delta}} \one\big\{B^{i,j}_{(k-1)\delta}\in \ux^{\delta,-}_{(k-1)\delta}\hbox{ and } B^{i,j}_{(k-1)\delta}\ge a\big\}\le N\Bigr\}
\nn\\   
\ux^{\delta,-}_{k\delta}&\defi  \big\{B^{i,j}_{k\delta}: B^{i,j}_{(k-1)\delta} \in \ux^{\delta,-}_{(k-1)\delta}\hbox{ and }  B^{i,j}_{(k-1)\delta}\ge L^{\delta,-}_{(k-1)\delta}\big\}.\label{2-3}
\end{align}
Since entire families are cut at time $(k-1)\delta$, it is not always possible to keep exactly $N$ particles at time $k\delta$.
Hence, for fixed $\delta$, the number of particles in $\ux^{\delta,-}_{k\delta}$ is $N-O(1)$, where $O(1)$ is non-negative and its law converges as $N\to\infty$ to the law of the age of a renewal process with inter-renewal intervals distributed as $N_\delta$, the one-particle family size at time $\delta$. The age law is the size-biased law of $N_\delta$. Since $N_\delta$ has all moments finite, this implies that $O(1)/N$ goes to zero almost surely and in $L_1$.  

We have the following expression for the barriers as a function of the ranked BBM $\cB$:
\begin{align}
  \ux^{\delta,-}_{k\delta} (\cB)
&= \big\{ B^{i,j}_{k\delta}: B^{i,j}_{\ell\delta}\ge L^{N,\delta,-}_{\ell\delta}, 0\le \ell\le k-1\big\}\label{u45}\\
\ux^{\delta,+}_{k\delta} (\cB)
&= \big\{ B^{i,j}_{k\delta}: B^{i,j}_{\ell\delta}\ge L^{N,\delta,+}_{\ell\delta}, 1\le \ell\le k\big\}.\nn
\end{align}


\paragraph{Partial order and domination.}
Let $\ux$ and $\uy$ be finite particle configurations and define
\begin{align}
\label{porder}
  \ux\preccurlyeq \uy \quad  {\text{if and only if}}\quad   | \ux\cap  [a,\infty) | \le   | \uy\cap  [a,\infty) | \quad \forall a\in\R.
\end{align}
In this case, we say that $\ux$ is dominated by $\uy$.
In Section \ref{s3}, we prove  the following dominations.  
\begin{proposition}
\label{domination}
For each $\delta>0$, there exists a coupling $\big(( \hux^{\delta,-}_{k\delta},\hux_{k\delta},\hux^{\delta,+}_{k\delta}):k\ge 0\big)$, whose marginals  have respectively the same law as $(\ux^{\delta,-}_{k\delta}:k\ge0)$, $(\ux_{k\delta}:k\ge0)$ and $(\ux^{\delta,+}_{k\delta}:k\ge 0)$, such that 
\begin{align}
\label{domination1}
  \hux^{\delta,-}_{k\delta}\preccurlyeq\hux_{k\delta}\preccurlyeq\hux^{\delta,+}_{k\delta},\quad k\ge 0.
\end{align}
\end{proposition}
Under this coupling, $\hux_t=\ux_t$; nevertheless, in order to maintain the dominations, $\hux^{\delta,\pm}_t$ are functions of the ranked BBM's $\cB^\pm$, which does not coincide with $\cB$ but have the same law. 
\paragraph{Deterministic barriers.}
Take $u\in L^1(\mathbb R,\mathbb R_+)$. The Gaussian kernel $G_t$ is defined by 
\[
G_tu(a)\defi  \int_{-\infty}^\infty \frac1{\sqrt{2\pi t}} e^{-(a-r)^2/{2t}}u(r)\,dr,
\]
so that $e^tG_t\ro$ is solution of the equation $u_t=\frac 12 u_{rr}+u$ with initial condition $\ro$.
For $m>0$, the \emph{cut operator} $C_m$ is defined by
	\begin{equation}
C_m u(a)\defi u(a) \mathbf 1\llav{\int_a^{\infty} u(r) dr < m},
\label{cut1}
	\end{equation}
so that $C_mu$ has total mass $\|u\|_1\wedge m$.
For $\delta>0$ and $k\in\N$, define the upper and lower barriers  $S^{\delta,\pm}_{k\delta}\ro$ at time $\delta k$ as follows:
\begin{align}
S^{\delta,\pm}_{0}\ro\defi \ro ;\qquad 
    S^{\delta,+}_{k\delta}\ro\defi \pare{C_{1}e^\delta G_\delta }^k\ro;\qquad S^{\delta,-}_{k\delta}\ro\defi \pare{e^\delta G_\delta C_{e^{-\delta}}}^k\ro.  \label{dpm1}
  \end{align}
To obtain the upper barrier $S^{\delta,+}_{\delta}\ro$, first diffuse\&grow for time $\delta$, and then cut mass from the left to keep mass 1.	 
To get the lower barrier $S^{\delta,-}_\delta\ro$, first cut mass from the left to keep mass $e^{-\delta}$, and then diffuse\&grow for time $\delta$.
Iterate to get the barriers at times $k\delta$. Since $\|e^\delta G_\delta u\|=e^\delta\|u\|$, we have 
$\big\|S^{\delta,\pm}_{k\delta}\ro\big\|_1= \norm{\ro}_1=1$ for all $k$. 

\paragraph{Hydrodynamics of $\delta$-barriers.}
In Section \ref{s4}, we  prove that, for fixed $\delta$, the empirical measures converge as $N\to\infty$ to the macroscopic barriers:
\begin{theorem}
	\label{thm4}
Let $\pi^{N,\delta,\pm}_{k\delta}$ be the empirical measures associated to the stochastic barriers $\ux^{\delta,\pm}_{k\delta}$ with initial configuration $X_0$. 
Then,
for any $a\in\R$, $\delta>0$ and $k\in\N$,
\begin{equation*}
\lim_{N\to \infty} \int_a^\infty \pi_{k\delta}^{N,\delta,\pm}(dr) = 
 \int_a^\infty  S^{\delta,\pm}_{k\delta}\ro(r)dr = 0,\quad \hbox{a.s. and in }L^1.
   \end{equation*}
The same is true if we substitute  $\ux^{\delta,\pm}_{k\delta}$ by the coupling marginals $\hux^{\delta,\pm}_{k\delta}$ of Proposition \ref{domination}.
\end{theorem}
\paragraph{Convergence of macroscopic barriers.} For $u,v\in L^1\pare{\R,\R^+}$, we write
\[
u \preccurlyeq v\quad  {\text{iff}}\quad \int_a^\infty u(r)dr\le \int_a^\infty v(r)dr\quad \forall a\in\R.
\]
In Section \ref{s5}, we fix $t$ and take $\delta=t/2^{n}$ to prove that, for the order $\preccurlyeq$,  the sequence $S^{t/2^n,-}_{t}\ro$ is increasing, the sequence $S^{t/2^n,+}_{t}\ro$ is decreasing, and  $\|S^{t/2^n,+}_{t}\ro-S^{t/2^n,-}_{t}\ro\|_1\to 0$ as $n\to\infty$. As a consequence,  we get the following theorem.
\begin{theorem}
  \label{thm5}
There exists a continuous function called $\psi(r,t)$ such that, for any $t>0$,
\begin{align*}
  \lim_{n\to\infty} \|S^{t/2^n,\pm}_{t}\ro -\psi(\cdot,t)\|_1 = 0.
\end{align*}
\end{theorem}
\paragraph{Sketch of proof of Theorems \ref{hydrodynamics} and  \ref{u-existence}.} The coupling of Proposition \ref{domination} satisfies $\hux^{\delta,-}_t\cle \hux_t\cle \hux^{\delta,+}_t$. 
By Theorem \ref{thm4}, the empirical measures associated to the stochastic barriers $\hux^{\delta,\pm}_t$ converge to the macroscopic barriers $S^{\delta,\pm}_t\ro$. The macroscopic barriers converge to a function $\psi(\cdot,t)$, as $\delta\to0$,  by Theorem \ref{thm5}. Hence the empirical measure of $\hux_t$ must converge to $\psi(\cdot,t)$ as $N\to\infty$. This is enough to get Theorem \ref{hydrodynamics}.

In Section  \ref{s6}, we show that any solution of the free boundary problem is in between the barriers $S_{k\delta}^{\delta,\pm}\ro$; this is enough to get Theorem \ref{u-existence}.

%
%
%
%
%
%
%
%


\section{Domination. Proof of Proposition \ref{domination}}
\label{s3}
\paragraph{Pre-selection inequalities.}
Recall that $(i,j)$ is the rank of particle $B^{i,j}_t$, and define the \emph{rank order} 
\begin{align}
  (i,j)\prec (i',j') \hbox{ if and only if }B^{i,1}_0<B^{i',1}_0\hbox{ or }i=i'\hbox{ and }j<j'.
\end{align}
The \emph{rank-selected} $N$-BBM consists on the positions of the $N$ particles with highest ranks at time $t$, denoted by
\begin{align}
\label{ytbt}
  Y_t(\cB):= \big\{B^{i,j}_t: \big|\{B^{i',j'}_t:(i,j)\prec(i',j')\}\big| < N\big\}.
\end{align}
Despite this is not a Markov process, one can produce a Markovian one by keeping track of the particle ranks as follows.
Between branching times particles move according to independent Brownian motions.
At each branching time, the particle with smallest rank jumps to the newborn particle, adopting its family and updating conveniently its rank. 

Recall the definition \eqref{zbt} of $\uz_t$ and observe that $\uy_t$ is a subset of $\uz_t$ with exactly $N$ particles of the rightmost families at time 0, while $\ux^{\delta,-}_\delta$ consists of the descendance at time $\delta$ of the maximal possible number of rightmost particles at time 0 whose total descendance at time $\delta$ does not exceed $N$. Hence $\ux^{\delta,-}_\delta \subset \uy_\delta$, which in turn implies
\begin{align}
\label{x-my}
\ux^{\delta,-}_\delta  \cle \uy_\delta.
\end{align}

\emph{Coupling}. We define a coupling between two vectors
\begin{align}
  (\ux_t^1,\dots,X_t^N)\hbox{ and }\big((\uy_t^1,\sigma_t^1),\dots,(\uy_t^N,\sigma_t^N)\big)
\end{align}
with $\ux^\ell_t,\uy^\ell_t\in \R$ and ranks $\sigma^\ell_t\in  \{1,\dots,N\}\times\{1,2,\dots\}$.

The particles stay spatially ordered $\uy^\ell_t\le \ux^\ell_t$ for all $\ell $ and $t$, and the sets of positions of $\ux^\ell_t$ and $\uy^\ell_t$ will have the same law than the processes $\ux_t$ and $\uy_t$, respectively.

Start with initial positions $(\ux^1_0,\dots,\ux^N_0)$ and $(\uy^1_0,\dots,\uy^N_0)$ satisfying $\uy^\ell_0\le \ux^\ell_0$ for all $\ell$.
If we have less than $N$ particles in the $\uy$-vector, we pretend that there are extra $\uy^\ell$-particles located at $-\infty$ and define the coupling for two vectors with $N$ coordinates each, anyway.
Let $\sigma^\ell_0$ be the ordered rank of the $\ell$-th $\uy$-particle defined by $\sigma_0^\ell\in \{(1,1),\dots,(N,1)\}$ and $\sigma^\ell_0\prec \sigma^{\ell'}_0$ if and only if $\uy^\ell_0<\uy^{\ell'}_0$ (we are assuming all the points are different; if not, use any criterion to break the ties).

Let $\ux_t$ be the $N$-BBM starting from the set $\{\ux^1_0,\dots,\ux^N_0\}$.
Between branching times, $\ux^\ell_t$ follows the increment of some Brownian particle.
At each branching time the left-most $\ux^\ell$-particle jumps to the branching place and start following the increments of the newborn particle.

Between branching events, $\ux_t^\ell-\uy_t^\ell$ and $\usi_t^\ell$ are constant, that is, the increments of $\uy_t^\ell$ clone those of $\ux^\ell_t$ ---independent Brownian motions--- and the ranks do not change.

Let $s$ be a branching time for the $\ux^n$-particle and suppose we have established $\sigma^\ell_{s-}, \uy^\ell_{s-},  \ux^\ell_{s-}$ for all $\ell$. Let $m$ be the label of the left-most $\ux^\ell$-particle before the branching, that is,  $\ux^m_{s-} = \min_{\ell}  \ux^\ell_{s-}$.
Let $h$ be the label of the minimal ranked $\uy^\ell$-particle, that is $\sigma^h_{s-} \prec\sigma^{\ell}_{s-}$ for all $\ell\neq h$. There are two cases. 

(1) $n\notin\{m,h\}$. 
At time $s$, set $\ux^m_s= \ux^n_{s-}$ and let the remaining $\ux^\ell$-particles keep their positions: $\ux^\ell_s = \ux^\ell_{s-}$ for all $\ell\neq m$. That is, the $\ux^m$-particle jumps to the newborn particle at~$\ux^n_{s-}$ and starts to follow its increments.

The positions and ranks of the $\uy^\ell$-particles are modified as follows.
Assume $\sigma^n_{s-}= (i,j)$.
If $\pare{i,M_t^i}=\max_{j'}\sigma_{t}(i,j')$ is the maximal rank  present at time $t$ in the $i$-family, set $\sigma^m_s= (i, 1+M^i_{s-})$,  $\uy^m_s = \uy^n_{s-}$.
We have two sub-cases.

(1a) If $h=m$, let the remaining particles keep their positions and ranks, that is, $(\uy^\ell_s,\sigma^\ell_s)=(\uy^\ell_{s-},\sigma^\ell_{s-})$ for $\ell \neq m$.

(1b) If $h\neq m$, set $\uy^h_s = \uy^m_{s-}$, $\sigma^h_s=\sigma^m_{s-}$ and let the remaining particles keep their positions and ranks. See Figure 1. 

(2) $n\in\{m,h\}$. Update the rank $\sigma^h_{s}= (i, 1+M^i_{s-})$, keep the other ranks and consider the following subcases. 

(2a) If $n=h \ne m$ set $\ux^m_s = \ux^h_{s-}$ and $\uy^h_s= \uy^m_{s-}$ and let the remaining particles keep their positions. See Figure 2.

(2b) If $n= m\ne h$ let all particles keep their positions.

(2c) If $n=h=m$ let all particles keep their positions.



\begin{figure}[h]
\centering
\hspace{-70pt}
\begin{minipage}[th]{.2\textwidth}
  \centering
  \begin{tikzpicture}[thick,scale=0.9,every node/.style={transform shape}]
\draw (-7.013,2.7) node[anchor=north west,text width=10pt] {before  jumps};
\draw (-7.013,0.18) node[anchor=north west,text width=10pt] {after jumps};
\draw (-7.013,-5) node[anchor=north west,text width=10pt] {};
\end{tikzpicture}
\end{minipage}%
\hspace{-15pt}
\begin{minipage}[t]{.4\textwidth}
  \centering
\begin{tikzpicture}[thick,scale=0.9,every node/.style={transform shape}]
\draw[->,line width=.8pt,domain=-2.9:0.4,smooth,variable=\x]
plot ({\x},{ -0.1*(\x+3)*(\x-0.5)+1});
\draw[<-,line width=.8pt,domain=-2.9:-1.1,smooth,variable=\x]
plot ({\x},{ 0.25*(\x+3)*(\x+1)+.8});
\draw[->,line width=.8pt,domain=-1.9:1.9,smooth,variable=\x]
plot ({\x},{ -0.17*(\x+2)*(\x-2)+1.7});
\draw [line width=1.2pt] (-4,1.6) -- (3.3,1.6);
\draw [line width=1.2pt] (-4,.9) -- (3.3,.9);
\draw [line width=1.2pt] (-4,-1.6) -- (3.3,-1.6);
\draw [line width=1.2pt] (-4,-0.9) -- (3.3,-.9);
\draw (-2.3,2.3) node[anchor=north west] {$m$};
\draw (1.92,2.2) node[anchor=north west] {$n$};
\draw (-0.54,2.22) node[anchor=north west] {$h$};
\draw (-3.6,.87) node[anchor=north west] {$m$};
\draw (-1.14,.8) node[anchor=north west] {$h$};
\draw (0.475,.87) node[anchor=north west] {$n$};
\draw (1.55,-.99) node[anchor=north west] {$n  \ m$};
\draw (-0.8,-.27) node[anchor=north west] {$h$};
\draw  (.05,-1.75) node[anchor=north west] {$n  \ m$};
\draw (-3.26,-1.67) node[anchor=north west] {$h$};
\draw[fill]   (-2,1.6) circle (2.2pt);
\draw[fill]   (-0.3,1.6) circle (2.2pt);
\draw[fill]   (2,1.6) circle (2.2pt);
\draw[fill]   (-3,0.9) circle (2.2pt);
\draw[fill]   (-1,0.9) circle (2.2pt);
\draw[fill]   (0.5,0.9) circle (2.2pt);
\draw[fill]   (2,-.8) circle (2.2pt);
\draw[fill]   (2,-1) circle (2.2pt);
\draw[fill]   (-0.3,-.9) circle (2.2pt);
\draw[fill]   (.5,-1.5) circle (2.2pt);
\draw[fill]   (-3,-1.6) circle (2.2pt);
\draw[fill]   (.5,-1.7) circle (2.2pt);
\draw[line width=0.8pt,dotted] (-3,.9)--(-3,-1.6);
\draw[line width=0.8pt,dotted] (-.3,1.6)--(-.3,-.8);
\draw[line width=0.8pt,dotted] (.5,.9)--(.5,-1.6);
\draw[line width=0.8pt,dotted] (2,1.6)--(2,-.8);
\end{tikzpicture}
  \captionof{figure}{\small Relative positions of particles at branching time $s$ for the case (1b). Each $X$-particle is to the right of the $Y$-particle with the same label before and after the branching. This order would be broken if the $h$th $Y$-particle jumped to the $n$th $Y$-particle, in this example.}
  \label{fig:test1}
\end{minipage}%
\ \ \ \ \ 
\begin{minipage}[t]{.4\textwidth}
  \centering
\begin{tikzpicture}[thick,scale=0.9,every node/.style={transform shape}]
\draw[<-,line width=.8pt,domain=-2.9:0.4,smooth,variable=\x]
plot ({\x},{ -0.125*(\x+3)*(\x-0.5)+1});
\draw[->,line width=.8pt,domain=-1.9:1.9,smooth,variable=\x]
plot ({\x},{ -0.17*(\x+2)*(\x-2)+1.7});
\draw [line width=1.2pt] (-4,1.6) -- (3.3,1.6);
\draw [line width=1.2pt] (-4,.9) -- (3.3,.9);
\draw [line width=1.2pt] (-4,-1.6) -- (3.3,-1.6);
\draw [line width=1.2pt] (-4,-0.9) -- (3.3,-.9);
\draw (-2.5,2.2) node[anchor=north west] {$m$};
\draw (1.937,2.27) node[anchor=north west] {$n=h$};
\draw (-3.02,0.83) node[anchor=north west] {$n=m$};
\draw (0.475,.87) node[anchor=north west] {$h$};
\draw (1.53,-1) node[anchor=north west] {$h \, \ m$};
\draw (-3.47,-1.67) node[anchor=north west] {$h \, \ m$};
\draw (3.4,2) node[anchor=north west] {\large $X_{s-}$};
\draw (3.4,1.28) node[anchor=north west] {\large $Y_{s-}$};
\draw (3.4,-.51) node[anchor=north west] {\large $X_s$};
\draw (3.4,-1.23) node[anchor=north west] {\large $Y_s$};
\draw[fill]   (-2,1.6) circle (2.2pt);
\draw[fill]   (2,1.6) circle (2.2pt);
\draw[fill]   (-3,0.9) circle (2.2pt);
\draw[fill]   (0.5,0.9) circle (2.2pt);
\draw[fill]   (2,-.8) circle (2.2pt);
\draw[fill]   (2,-1) circle (2.2pt);
\draw[fill]   (-3,-1.5) circle (2.2pt);
\draw[fill]   (-3,-1.7) circle (2.2pt);
\draw[line width=0.8pt,dotted] (-3,.9)--(-3,-1.6);
\draw[line width=0.8pt,dotted] (2,1.6)--(2,-.8);
\end{tikzpicture}
\vspace{-14.3pt}
  \captionof{figure}{\small Case (2a). When $n=m$ only the $h$-th $Y$-particle jumps to $Y^n_{s-}$ while when $n=h$ only the $m$-th $X$-particle jumps to $X^{h}_{s-}$. We perform these two cases simultaneously.}
  \label{fig:test2}
\end{minipage}
\end{figure}
\emph{Remarks. } In case (1a), the $m$-th particle of each process goes to the respective new-born particle at positions $\ux^n_{s-}$ and $\uy^n_{s-}$ respectively. In case (1b) the $m$-particle goes to the new-born particle at $\uy^n_{s-}$ and the $m$-rank takes a rank in the family of the newborn particle while the $h$-th rank and particle take the rank and position of the $\uy^m$-particle. When $n\neq m$ in case (2a) we couple the branching of the $\ux^h$-particle with the branching of the $\uy^m$-particle while in case (2b) we couple the branching of the $\ux^m$-particle with the branching of the $\uy^h$-particle. None branching produce a new particle but we keep track of this time by increasing the rank of the $Y^h$-particle. In case (2c) both $\ux^m$ and $\uy^h$ particles branch but neither produce a new particle; we also keep track of this time by increasing the rank of the $Y^h$-particle. The instructions of the two last cases are the same but in (2b) we couple particles with different labels while in (2c) we couple particles with the same label;  so we simply are stressing this difference. We introduce the counters $M^i_t$ to rank newborn $\uy^\ell$-particles and track branching times of $\uy^\ell$-particles that have not produced new particles. Observe that $M^i_t$ also tracks the size of the $i$-family until its lower ranked members start jumping to families with higher ranks.

 By construction, we have that  $\ux_t= \{\ux^1_t,\dots,\ux^N_t\}$. Denote the set of positions occupied by the $\uy^\ell$-particles by 
\begin{align}\label{huytl}
  \huy_t:= \{\uy^1_t,\dots,\uy^N_t\},\quad t\ge0.
\end{align}
Let us stress that $\huy_t$ is also a function of $\cB$ but does not coincide with $\uy_t(\cB)$ given in \eqref{ytbt}.

\begin{lemma}
The  coupling marginals satisfy (a) $\ux_t$ is the $N$-\nobreak{\rm BBM} and  (b) $\huy_t$ has the distribution of $\uy_t$, the rank-selected process described at the beginning of this section. Moreover, 
\begin{align}
\label{ymx}
  \huy_t\cle \ux_t.
\end{align}
\end{lemma}
\begin{proof}  
By construction, the labels of the $\ux^\ell$-processes were added to a realization of the $N$-BBM $\ux_t$ given by \eqref{xt1}. So, the unlabeled positions of $\ux^\ell$-particles coincide with  $\ux_t$.

Disregarding the labels, we see that, at each branching event, the lower ranked particle of the leftmost present $\uy^\ell$-family jumps to the branching $\uy^\ell$-particle. Those branchings occur at rate~1 because they copy the branchings of the $\ux^\ell$-particles. Since  $\uy^\ell$-particle increments clone the (Brownian) $\ux^\ell$-particle increments, we have that the positions of the $\huy^\ell$-marginal of the coupling has the same distribution as $\uy_t$ (see remark after \eqref{ytbt}).

The domination $\uy^\ell_t\le \ux^\ell_t$ holds at time 0 and it is preserved between branching events because the Brownian increments are the same. The domination persists after each branching event obviously in cases (1a), (2b) and (2c).
In case (1b), we have $\uy^m_s = \uy^n_{s-}\le \ux^n_{s-} =\ux^m_s$ and $\uy^h_s= \uy^m_{s-}\le \ux^m_{s-}\le \ux^h_{s-}=\ux^h_s$ because $\ux^m_{s-}$ is the minimal $\ux_{s-}$-particle and the $h$-th $\ux$ particle does not jump at $s$.
In case (2a), $\ux^m_s=\ux^h_{s}=\ux^h_{s-} \ge \ux^m_{s-} \ge \uy^m_{s-}= \uy^h_s=\uy^m_s$ (see Figure 2).
Iterating over all branching events occurring before $t$, we get $\uy^\ell_t\le \ux^\ell_t$ and \eqref{ymx}. 
\end{proof}

\paragraph{The post-selection process.}
We construct an auxiliary process $(V^1_t,\dots,V^N_t)$, with initial ordered positions $X^\ell_0\le V^\ell_0$, and such that, if $\ux^\ell_0=B^{i,1}_0$, then  $V^\ell$ follows the branching events and Brownian increments of the $i$ family but displacing the initial point to $V^\ell_0$. Then,  at time $t$ selects the $N$ rightmost positions of the offsprings of all $V^\ell$ particles:
\begin{align}
\label{388}
  V_t = \hbox{$N$ rightmost particles of the set } \cup_{\ell=1}^N\{V^\ell_0-X^\ell_0+ B^{i,j}_t : B^{i,1}_0=X^\ell_0\}. 
\end{align}
Since $\ux_t$ is dominated by the $N$ rightmost particles of $\uz_t$ and in turn these are dominated by $V_t$, we have 
\begin{align}
\label{389}
  \ux_t\cle V_t.
\end{align}
In particular, if $V^\ell_0=\ux^\ell_0$ for all $\ell$, we have $V_\delta= \ux^{\delta,+}_\delta$. 
Hence, with this coupling, 
\begin{align}
\label{xmx+}
  \ux_\delta\cle\ux^{\delta,+}_\delta.
\end{align}
 
\paragraph{Proof of Proposition  \ref{domination}.}
At time $\delta$, the first inequality in \eqref{domination1} follows from \eqref{x-my} and \eqref{ymx}, while  the second inequality is \eqref{xmx+}. To iterate the upperbound, assume $\ux_{k\delta} \cle \hux^{\delta,+}_{k\delta}$ and let $V^\ell_{k\delta}$ be the $\ell$-th  $\hux^{\delta,+}_{k\delta}$ particle such that $\ux_{k\delta}^\ell \le V^\ell_{k\delta}$ for all $\ell$. Use the construction \eqref{388} such that $V_{(k+1)\delta}$ consist on the $N$ rightmost particles of a BBM starting at $V_{k\delta}$ obtained by cloning the increments and branchings of the BBM  containing the trajectory of $\big(X_t:t\in[k\delta,(k+1)\delta)\big)$.
The inequality   $\ux_{(k+1)\delta}\cle V_{(k+1)\delta}$ holds by \eqref{389}. Define $\hux^{\delta,+}_{(k+1)\delta}:= V_{(k+1)\delta}$. The process  $(\hux^{\delta,+}_{k\delta}:k\ge 0)$ so defined has the same distribution as $\ux^{\delta,+}_{k\delta}:k\ge0)$. 

To iterate the lower bound, let $\huy^\delta_\delta$ be defined by the coupling for $t=\delta$. Let $\huy^\delta_t$ for $t\in[k\delta, (k+1)\delta)$ be defined by the initial condition $\huy^\delta_{k\delta}= \hux^{\delta,-}_{k\delta}$ and be governed by the increments and branching times of the $\hux_t$-particles for $t$ in that interval, as in the coupling. Let $\hux^{\delta,-}_{(k+1)\delta}$ be the offsprings of the $\hY^\delta$ families that have not lost any member. As a consequence we get  $\hux^{\delta,-}_{(k+1)\delta}\cle\hux_{(k+1)\delta}$. \qed

\paragraph{Barriers embedded on BBM} 

\begin{proposition}
  \label{ymx1}
There exist a coupling $ (\cB^-,\cB,\cB^+) $ with marginals distributed as the ranked {\rm BBM}  $ \cB$ such that 
the processes $\hux^{\delta,\pm}_{k\delta}$ defined by
\begin{align}
\ux_t &=\ux_t(\cB),\nn\\
  \hux^{\delta,-}_{k\delta} &:= \ux^{\delta,-}_{k\delta} (\cB^-),\nn\\
\hux^{\delta,+}_{k\delta} &:= \ux^{\delta,-}_{k\delta} (\cB^+),\nn
\end{align}
see \eqref{u45}, satisfy the order $ \hux^{\delta,-}_{k\delta}\cle\ux_{k\delta}\cle\hux^{\delta,+}_{k\delta}$.
\end{proposition}
\emph{Sketch proof.} 
  Take the coupled barriers  $\hux^{\delta,\pm}_{k\delta}$ and the auxiliary process $\huy^\delta_t$ coupled to $\hux_{k\delta}$ in Proposition  \ref{domination} and its proof. To construct $\cB^+$, at each time $\ell\delta$ attach independent BBM to the particles killed at that time for the process $\hux^{\delta,+}_{k\delta}$. To construct $\cB^-$, at each branching time $s\in[k\delta,(k+1)\delta)$  for the auxiliary process $(\huy^\delta_t: t\in k\delta,(k+1)\delta))$ associated to that time interval, attach an independent BBM to the space-time point $(\huy^\delta_s,s)$. Then rank the resulting BBM process in such a way that $\huy^\delta_t$ is its rank-selected $N$-BBM process in each time interval. For this it suffices to arrange the ranks in such a way that when there are two branches at the same point and one of them must be erased from the $\huy^\delta$ process, then the branch belonging to the $\huy^\delta$ process gets a bigger rank than the other branch. We leave to the reader the details of the construction and the proof that those processes satisfy the conditions of the proposition. \qed

\section{Hydrodynamic limit for the barriers}
\label{s4}

In this section, we prove Theorem \ref{thm4}, namely that the  stochastic
barriers converge in the macroscopic limit $N\to \infty$ to the
deterministic barriers.
Recall that
$\ro $ is a probability density on $\mathbb R$ with  a left boundary  $L_0$, 
and the $N$-BBM  starting from  $\ux_0=(\ux^1_0,\dots,\ux^N_0)$, iid continuous random variables with density $\ro$. 

It is convenient to have a notation for the cutting points for the macroscopic barriers $S^{\delta,\pm}_{k\delta}$ defined in \eqref{dpm1}. For $\delta>0$ and natural number $\ell\le k$ denote 
\begin{align}
\label{lcp}
  L^{\delta,+}_{\ell\delta} := \inf_r\Bigl\{\int_r^{\infty} S^{\delta,+}_{\ell\delta}\ro(r')dr' < 1\Bigr\}; \qquad
L^{\delta,-}_{\ell\delta} := \inf_r\Bigl\{\int_r^{\infty} S^{\delta,-}_{\ell\delta}\ro(r')dr'< e^{-\delta} \Bigr\}.
\end{align}
Let $B_0$ be a continuous random variable with density $\ro$. Let $B_{[0,t]}=(B_s:s\in[0,t])$ be Brownian motion starting from $B_0$, with increments independent of $B_0$. Let $N_t$ be the random size at time $t$ of a BBM family starting with one member; we have $EN_t = e^t$. Recall $e^tG_t\ro$ is the solution of $u_t=\frac 12 u_{rr}+u$ with $u(\cdot,0)=\ro$. With this notation, we have the following representation of $e^tG_t\ro$ and the macroscopic barriers as expectation of functions of the Brownian trajectories.
\begin{lemma}
\label{th1}
For every test function $\varphi\in L^{\infty}(\mathbb R)$ and every $t>0$, we have
	 \begin{eqnarray}
    \label{sb1.1}
\int \varphi(r) e^tG_t\ro(r)\,dr = e^tE[\varphi(B_t)]
 \end{eqnarray}
Furthermore, 
\begin{align}
\label{ivs1}
 \int \varphi(r) S^{\delta,+}_{k\delta}\ro(r)\,dr &= e^{k\delta} E\big[\varphi(B_{k\delta})\one\big\{B_{\ell\delta} >L_{\ell\delta}^{\delta,+}:1\le \ell\le k\big\}\big].\\
\int \varphi(r) S^{\delta,-}_{k\delta}\ro(r)\,dr &= e^{k\delta} E\big[\varphi(B_{k\delta})\one\big\{B_{\ell\delta} >L_{\ell\delta}^{\delta,-}:0\le \ell\le k-1\big\}\big].
 \end{align}
\end{lemma}

\begin{proof} Immediate.
\end{proof}
Recall the definition of $\cB$ in \eqref{cB}, in particular the trajectory $B_{[0,t]}^{i,j}$ is distributed as $B_{[0,t]}$ for all $i,j$ and the families $(B^{i,j}_{[0,t]}:j\in\{1,\dots,N^i_t\})$, for $i\in\{1,\dots,N\}$ are iid.  
\begin{proposition} 
\label{propo12}
Let $g:C(\R^+,\R)\to\R$ be a bounded measurable function  and define
\begin{align}
  \mu^N_t g:= \frac1N\sum_{i=1}^N\sum_{j=1}^{N_t^i} g(B_{[0,t]}^{i,j}).
\end{align}
Then, 
    \begin{align}
\label{llnmu}
  \lim_{N\to\infty}\mu^N_tg  = e^t Eg(B_{[0,t]}),\qquad \hbox{a.s. and in }L^1.
  \end{align}
\end{proposition}
\begin{proof}
    By the many-to-one Lemma we have
  \begin{align}
   E  \mu^N_tg=  EN_t\, Eg(B_{[0,t]}) =  e^t Eg(B_{[0,t]}),
  \end{align}
Furthermore,  the variance of $\mu^N_tg$ is order $1/N$. Indeed, using family independence we get
\begin{align}
  {\rm Var}(\mu^N_tg)  &= \frac1N {\rm Var}\Bigl(\sum_{j=1}^{N_t^i}  g(B_{[0,t]}^{i,j})\Bigr)\le \frac1N E(N_t^i)^2 \|g\|_\infty^2.\nn
\end{align}
This is enough to get the strong law of large numbers \eqref{llnmu}.
\end{proof}
\begin{corollary}[Hydrodynamics of the BBM]
  \begin{align}
  \lim_{N\to\infty}  \frac1N\sum_{i=1}^N\sum_{j=1}^{N_t^i} \varphi(B^{i,j}_t ) =  e^tE\varphi(B_t)= e^t \int\varphi(r)G_t\ro(r)dr,\qquad \hbox{a.s. and in }L^1.
  \end{align}
\end{corollary}



By Definitions \eqref{2-2} and \eqref{2-3} of the microscopic cutting points $
  L^{N,\delta,\pm}_{\ell\delta}$  we have
\begin{align}
  \frac1N\sum_{i=1}^N\sum_{j=1}^{N_{k\delta}^i} \one\{B^{i,j}_{\ell\delta}\ge L^{N,\delta,+}_{\ell\delta}: 1\le\ell\le k\} &= 1\nn\\
\frac1N\sum_{i=1}^N\sum_{j=1}^{N_{k\delta}^i} \one\{B^{i,j}_{\ell\delta}\ge L^{N,\delta,-}_{\ell\delta}: 0\le\ell\le k-1\} &=  1 - O(1/N), \label{4-9}
\end{align}
see paragraph after \eqref{2-3}. Let
\begin{align}
 A_{k\delta} ^{N,\delta,+}:=\prod_{\ell=1}^k \one\big\{ B^{i,j}_{\ell\delta}\ge L^{N,\delta,+}_{\ell\delta}\big\};\qquad
&A_{k\delta} ^{\delta,+}:=\prod_{\ell=1}^k \one\big\{ B^{i,j}_{\ell\delta}\ge L^{\delta,+}_{\ell\delta}\big\}\nn\\
 A_{k\delta} ^{N,\delta,-}:=\prod_{\ell=0}^{k-1} \one\big\{ B^{i,j}_{\ell\delta}\ge L^{N,\delta,-}_{\ell\delta}\big\};\qquad
&A_{k\delta} ^{\delta,-}:=\prod_{\ell=0}^{k-1} \one\big\{ B^{i,j}_{\ell\delta}\ge L^{\delta,-}_{\ell\delta}\big\}\nn
\end{align}
These quantities depend on $i,j$ but we suppress them in the notation. 
\begin{proposition}
  \label{llnlt}
We have
  \begin{align}
  \label{llnlt1}
  \frac1N\sum_{i=1}^N\sum_{j=1}^{N_{k\delta}^i}  
\Bigl| A_{k\delta} ^{N,\delta,\pm}- A_{k\delta} ^{\delta,\pm}
\Bigr|
\mathop{\longrightarrow}_{N\to\infty}0 ,\quad a.s. \hbox{ and in }L^1.
  \end{align}
\end{proposition}
\begin{proof}
  Since at time zero the families have only one element, for the lower barrier at $k=0$ the left hand side of \eqref{llnlt1} reads
  \begin{align}
    &\hskip-3mm\frac1N\sum_{i=1}^N
\Bigl|
\one\big\{ B^{i,1}_{0}\ge L^{N,\delta,-}_{0}\big\}
- \one\big\{ B^{i,1}_{0}\ge L^{\delta,-}_{0}\big\}
\Bigr|.\nn
\end{align}
Recalling that all the trajectories $B^{i,j}_{[0,\delta]}$ start at the same point $B^{i,1}_0$, we can bound the above expression by 
\begin{align}
 &\frac1N\sum_{i=1}^N\sum_{j=1}^{N_\delta^i}
\Bigl|\one\big\{ B^{i,j}_{0}\ge L^{N,\delta,-}_{0}\big\}
- \one\big\{ B^{i,j}_{0}\ge L^{\delta,-}_{0}\big\}\Bigr| \label{llnlt2} \\
&\quad= \Bigl|\frac1N\sum_{i=1}^N\sum_{j=1}^{N_\delta^i}
\one\big\{ B^{i,j}_{0}\ge L^{N,\delta,-}_{0}\big\}
- \frac1N\sum_{i=1}^N\sum_{j=1}^{N_\delta^i}\one\big\{ B^{i,j}_{0}\ge L^{\delta,-}_{0}\big\}
\Bigr|
\mathop{\longrightarrow}_{N\to\infty}0,  \hbox{ {\rm a.s.} and $L^1$}\label{llnlt6}.
  \end{align}
The identity holds because the differences of the indicator functions in \eqref{llnlt2} have the sign of $L^{N,\delta,-}_{0}-L^{\delta,-}_{0}$ for all $i,j$. The limit \eqref{llnlt6} holds because (a) by \eqref{4-9} the first term in \eqref{llnlt6} converges to~1 and (b) the limit of the second term is also 1 by definition \eqref{lcp} of $L^{\delta,-}_0$ and Proposition~\ref{propo12}.

For the upper barrier at  $k=1$, put again the sums between the absolute values to get that in this case the left hand side of \ref{llnlt1} reads
\begin{align}
\Bigl|
\frac1N\sum_{i=1}^N\sum_{j=1}^{N_{\delta}^i}   \one\big\{ B^{i,j}_{\delta}\ge L^{N,\delta,+}_{\delta}\big\}
- \frac1N\sum_{i=1}^N\sum_{j=1}^{N_{\delta}^i}  \one\big\{ B^{i,j}_{\delta}\ge L^{\delta,+}_{\delta}\big\}
\Bigr| \mathop{\longrightarrow}_{N\to\infty}0 ,\quad a.s. \hbox{ and in }L^1,\label{llnlt7}
\end{align}
because the first term is 1 by definition of $L^{N,\delta,+}_\delta$ and the second one converges to 1 by the definition of $L^{\delta,+}$ and Proposition \ref{propo12}.

For the induction step assume \eqref{llnlt1} holds for $\ell=k-1$ and write the left hand side of \eqref{llnlt1} for the upper barrier at $\ell=k$ as 
\begin{align}
   & \frac1N\sum_{i=1}^N\sum_{j=1}^{N_{k\delta}^i}  
\Bigl| A_{(k-1)\delta} ^{N,\delta,+} \one\big\{ B^{i,j}_{k\delta}\ge L^{N,\delta,+}_{k\delta}\big\}- A_{(k-1)\delta} ^{\delta,+}\one\big\{ B^{i,j}_{k\delta}\ge L^{\delta,+}_{k\delta}\big\}\Bigr|\nn\\
&\qquad \le \Bigl|\frac1N\sum_{i=1}^N\sum_{j=1}^{N_{k\delta}^i}  
 A_{(k-1)\delta} ^{N,\delta,+} \one\big\{ B^{i,j}_{k\delta}\ge L^{N,\delta,+}_{k\delta}\big\}- \frac1N\sum_{i=1}^N\sum_{j=1}^{N_{k\delta}^i}  A_{(k-1)\delta} ^{N,\delta,+}\one\big\{ B^{i,j}_{k\delta}\ge L^{\delta,+}_{k\delta}\big\}\Bigr|  \label{llnlt3}\\
&\qquad\qquad+\frac1N\sum_{i=1}^N\sum_{j=1}^{N_{k\delta}^i}  
\Bigl| A_{(k-1)\delta} ^{N,\delta,+} \one\big\{ B^{i,j}_{k\delta}\ge L^{\delta,+}_{k\delta}\big\}- A_{(k-1)\delta} ^{\delta,+}\one\big\{ B^{i,j}_{k\delta}\ge L^{\delta,+}_{k\delta}\big\}\Bigr|, \label{llnlt4}
\end{align}
where the inequality is obtained by summing and subtracting the same expression and then taking the modulus out of the sums in \eqref{llnlt3} as all the indicator function differences have the same sign, as before.  
By dominated convergence and the induction hypothesis, the expression in \eqref{llnlt4} converges to zero as $N\to\infty$. In turn, this implies that the second term in  \eqref{llnlt3} has the same limit as the second term in \eqref{llnlt5} below. Hence we only need to show that the limits of the following expressions vanish.
\begin{align}
\label{llnlt5}
  \Bigl|
\frac1N\sum_{i=1}^N\sum_{j=1}^{N_{k\delta}^i}   A_{(k-1)\delta} ^{N,\delta,+} \one\big\{ B^{i,j}_{k\delta}\ge L^{N,\delta,+}_{k\delta}\big\}
- \frac1N\sum_{i=1}^N\sum_{j=1}^{N_{k\delta}^i}  A_{(k-1)\delta} ^{\delta,+}\one\big\{ B^{i,j}_{k\delta}\ge L^{\delta,+}_{k\delta}\big\}
\Bigr| 
\end{align}
(informally, we have taken the modulus outside the sums). 
Those limits follow with the arguments used to show \eqref{llnlt7}, indeed the first term is 1 and the second one converges to~1. The same argument shows that the limit \eqref{llnlt1} for the lower barrier at $\ell=k$ is the same as the limit of the expression
\begin{align}
\label{llnlt8}
  \Bigl|
\frac1N\sum_{i=1}^N\sum_{j=1}^{N_{k\delta}^i}   A_{(k-1)\delta} ^{N,\delta,-} \one\big\{ B^{i,j}_{\delta}\ge L^{N,\delta,-}_{(k-1)\delta}\big\}
- \frac1N\sum_{i=1}^N\sum_{j=1}^{N_{k\delta}^i}  A_{(k-1)\delta} ^{\delta,-}\one\big\{ B^{i,j}_{(k-1)\delta}\ge L^{\delta,-}_{k\delta}\big\}
\Bigr| ,
\end{align}
which goes to 0 by the same argument as  \eqref{llnlt6}.
\end{proof}

\paragraph{Proof of Theorem \ref{thm4}} Recalling \eqref{u45} we have
  \begin{align}
    \pi^{N,\delta,+}_{k\delta}\varphi &= \frac1N\sum_{i=1}^N\sum_{j=1}^{N_{k\delta}^i} \varphi(B^{i,j}_{k\delta}) \one\{B^{i,j}_{\ell\delta}\ge L^{N,\delta,+}_{\ell\delta}: 1\le\ell\le k\}\nn\\
    \pi^{N,\delta,-}_{k\delta}\varphi &= \frac1N\sum_{i=1}^N\sum_{j=1}^{N_{k\delta}^i} \varphi(B^{i,j}_{k\delta}) \one\{B^{i,j}_{\ell\delta}\ge L^{N,\delta,-}_{\ell\delta}: 0\le\ell\le k-1\}.\nn
  \end{align}
We want to apply Proposition \ref{llnmu} but do not have an explicit expression for $E \pi^{N,\delta,\pm}_{k\delta}$ because of the random boundaries in the right hand side. But if we use instead the deterministic boundaries $L^{\delta,\pm}_{\ell\delta}$ by defining
\begin{align}
g^{+}_\varphi(B_{[0,k\delta]}) &:=  \varphi(B_{k\delta}) \one\{B_{\ell\delta}\ge L^{\delta,+}_{\ell\delta}: 1\le\ell\le k\}  \label{gbt+}\\
g^{-}_\varphi(B_{[0,k\delta]}) &:=  \varphi(B_{k\delta}) \one\{B_{\ell\delta}\ge L^{\delta,-}_{\ell\delta}: 0\le\ell\le k-1\}\label{gbt-}
\end{align}
By  \eqref{llnmu} we have
\begin{align}
  \lim_{N\to\infty} \mu^N_{k\delta}g^{\pm}_\varphi = e^{k\delta} Eg^{\pm}_\varphi(B_{[0,k\delta]}) =  \int \varphi(r) S^{\delta,\pm}_{k\delta}\ro(r)\,dr , \quad\hbox{by \eqref{ivs1}}.
\end{align}
To conclude it suffices to show that  $ \pi^{N,\delta,\pm}_{k\delta}\varphi- \mu^N_{k\delta}g^{\pm}_\varphi$, converges to 0. Write
\begin{align}
&\big|\pi^{N,\delta,-}_{k\delta}\varphi -\mu^N_{k\delta}g^-_\varphi\big|= \Bigl|\frac1N\sum_{i=1}^N\sum_{j=1}^{N_{k\delta}^i} \varphi(B^{i,j}_{k\delta}) \Bigl(
\prod_{\ell=0}^{k-1} \one\big\{ B^{i,j}_{\ell\delta}\ge L^{N,\delta,-}_{\ell\delta}\big\}
-\prod_{\ell=0}^{k-1} \one\big\{ B^{i,j}_{\ell\delta}\ge L^{\delta,-}_{\ell\delta}\big\}\Bigr)\Bigr|\nn\\
&\quad\le \|\varphi\|_\infty \frac1N\sum_{i=1}^N\sum_{j=1}^{N_{k\delta}^i}\Bigl|
\prod_{\ell=0}^{k-1} \one\big\{ B^{i,j}_{\ell\delta}\ge L^{N,\delta,-}_{\ell\delta}\big\}
-\prod_{\ell=0}^{k-1} \one\big\{ B^{i,j}_{\ell\delta}\ge L^{\delta,-}_{\ell\delta}\big\}\Bigr|
\mathop{\longrightarrow}_{N\to\infty}0,  \hbox{ {\rm a.s.} and $L^1$},\nn
\end{align}
by Proposition \ref{llnlt}. The same argument works for the upper barrier. \qed

\section{Existence of the limit function $\psi$}
\label{s5}

In this Section we prove Theorem \ref{thm5}. Start with the following Proposition whose proof is similar to the one given in Chapters 4, 5 and 6 of \cite{MR3282863}.

\begin{proposition}\label{0hh}
The following properties hold for every $u,v\in L^1(\mathbb R,\mathbb R^+)$  and $t,m>0$.
\begin{itemize}
\item[(a)]
If $u\cle v$, then $C_m u\cle v$ for every $m>0$.
\item[(b)]
If $u\le v$ pointwise, then $G_tu\le G_tv$, $G_tu\cle G_tv$.
\item[(c)]
$C_m$ and $G_t$ preserve the order: if $u\cle v$, then $C_mu\cle C_mv$ and $G_tu\cle G_tv$.
\item[(d)]
$\norm{C_m u-C_m v}_1\le \norm{u-v}_1$.
\item[(e)]
$\norm{G_t u-G_t v}_1\le  \norm{u-v}_1$.
\item[(f)]
$\abs{\frac{\partial}{\partial r}G_t u(r)}\le \frac{c\norm{u}_\infty}{\sqrt t}$ for every $r\in\R$.
\end{itemize}
\end{proposition}

 \begin{proof}[Proof of Proposition \ref{0hh}]
 Items \textit{(a)}, \textit{(b)},  \textit{(e)} and  \textit{(f)} are simple and we omit their proofs.

\noindent{\it Proof of   \textit{(c)}}.  We start with $C_m$  and assume $m<\norm{u}_1\wedge \norm{v}_1$ as the other case is trivial.
 For $a\in\R$, we have to prove that $\int_a^{\infty}C_mu\le \int_a^{\infty}C_mv$.
 Denote the cutting points by
 \begin{align}\label{0sd}
 q_m(u)\defi \inf\Bigl\{a\in \mathbb R:\int_a^{\infty}u<m\Bigr\}
 \end{align}

 We suppose $q_m(u)\wedge q_m(v)<a<q_m(u)\vee q_m(v)$ as the other case is trivial.
 If $q_m(v)L_u<a<q_m(v)$, we have
 \begin{align*}
 \int_a^{\infty}C_mu 
 =\int_{a}^\infty u\le \int_{q_m(v)}^\infty u
 =m
 =\int_{q_m(v)}^\infty v
 =\int_{a}^\infty C_m v;
 \end{align*}
 if $q_m(v)<a<q_m(u)$,
 \begin{align*}
 \int_a^\infty C_mu = \int_{q_m(u)}^\infty u \le \int_{q_m(u)}^\infty v\le \int_{a}^\infty v
 =\int_a^\infty C_mv.
 \end{align*}


 Suppose now $\|u\|_1<\|v\|_1$ and let
 $m\defi \|v\|_1-\|u\|_1$.
 It is easy to see that $u\cle C_m v$.
 As $\|u\|_1=\|C_m v\|_1$, we can apply the previous case to get
 $G_tu\cle G_t C_m v$.
 We conclude by observing that, because of item (b) and the point-wise dominance $C_m v\le v$, we have $G_t C_m v\cle G_t v$.

\noindent{\it Proof of \textit{(d)}}.  
 We assume $m<\|u\|_1\wedge \|v\|_1$ as the other case is trivial.
 Define $q_m(u)$ and $q_m(v)$ as in \eqref{0sd} and suppose $q_m(u)\ge q_m(v)$ without loss of generality.
 Then
 \begin{equation}\label{0001}
 \|C_m u-C_m v\|_1=
 \norm{u\mathbf 1_{x\ge q_m(u)}-v\mathbf 1_{x\ge q_m(v)}}_1=\int_{q_m(v)}^{q_m(u)} v-
 \int_{-\infty}^{q_m(u)} |u-v|+  \|u-v\|_1.
 \end{equation}
 Also
 \begin{align}\label{0003}
 \int_{q_m(v)}^{q_m(u)} v= 
   \int_{-\infty}^{q_m(v)} v-\int_{-\infty}^{q_m(u)} u+\int_{q_m(v)}^{q_m(u)} v
 =\int_{-\infty}^{q_m(u)} (v-u)\le \int_{-\infty}^{q_m(u)} |u-v|.
 \end{align}
 Item \textit{(d)} follows from \eqref{0001} and \eqref{0003}.  \end{proof}
\begin{proposition}\label{thm031}
For every $\delta>0$ and natural number $k\ge 0$, we have
\begin{align}
\label{0eq1}
S_{k\delta}^{\delta,-}u & \cle S_{k\delta}^{\delta,+}u \\
\label{0eq2}
S_{k\delta}^{\delta,-}u & \cle S_{k\delta}^{\delta/2,-}u \\
\label{0eq3}
S_{k\delta}^{\delta,+}u & \cge S_{k\delta}^{\delta/2,+}u.
\end{align}
Furthermore, there exists a constant $c>0$ such that
\begin{align}
\label{0eq4}
\norm{S_{k\delta}^{\delta,+}u-S_{k\delta}^{\delta,-}u}_{1}\le c\delta.
\end{align}
\end{proposition}

 \begin{proof}[Proof of Proposition \ref{thm031}]
Inequality \eqref{0eq1} is a consequence of Theorem \ref{thm4} and Proposition \ref{domination}.

 To prove inequalities \eqref{0eq2} and \eqref{0eq3} we call $H_\delta^-\defi e^\delta G_\delta C_{e^{-\delta}}$ and $H^+_\delta\defi C_{1} e^\delta G_\delta$ and we  prove below that    \begin{align}\label{0011}
 H_{\delta}^-v & \cle \big(H_{\delta/2}^-\big)^2v
 \\
 \label{0012}
 H^+_{\delta}v&\cge \big(H^+_{\delta/2}\big)^{2}v
 \end{align} 
Now we deduce \eqref{0eq2} from \eqref{0011} by induction.  The case $k=1$ is \eqref{0011}. Assume the assertion holds for $k-1$ and
 write
 $\pare{H^-_{\delta}}^ku=H^-_{\delta}\pare{H^-_{\delta}}^{k-1}u$,
 call $v\defi\pare{H^-_{\delta}}^{k-1}u$ and use the case $k=1$ to infer $\pare{H^-_{\delta}}^ku\cle
 \big(H^-_{\delta/2}\big)^2\pare{H^-_{\delta}}^{k-1}u$.
By the inductive hypothesis and the fact that $\big(H^-_{\delta/2}\big)^2$ preserves the order conclude that $\pare{H^-_{\delta}}^ku\cle \big(H^-_{\delta/2}\big)^{2k}u$ which is \eqref{0eq2}. 
 The way to deduce \eqref{0eq3} from \eqref{0012} is similar.

 \noindent {\it {Proof of \eqref{0011} }}.  We first prove that for $a\in\mathbb R$
 \begin{align}\label{0007}
 \int_a^\infty e^{\delta/2} G_{\delta/2} C_{e^{-\delta}}v\le \int_a^\infty C_{e^{-\delta/2}}H_{\delta/2}^-v.
 \end{align} 
 Let   $q\defi q_{e^{-\delta/2}}(H_{\delta/2}^-v)$.  If $a\le q$, identity \eqref{0007} is satisfied because
 \begin{align}
 \int_a^\infty e^{\delta/2} G_{\delta/2} C_{e^{-\delta}}v
 \le \int_{-\infty}^\infty e^{\delta/2}G_{\delta/2} C_{e^{-\delta}}v
 =e^{-\delta/2}
 =
 \int_a^\infty C_{e^{-\delta/2}}H_{\delta/2}^-v.
 \end{align}
 If $a>q$, inequality \eqref{0007} becomes
 \begin{align}
 \int_a^\infty e^{\delta/2}G_{\delta/2} C_{e^{-\delta}}v\le
 \int_a^\infty H_{\delta/2}^-v,
 \end{align}
which follows from $C_{e^{-\delta}}v\cle C_{e^{-\delta/2}}v$ and  from (c) of Proposition \ref{0hh} applied to $e^{\delta/2}G_{\delta/2}$.
 
 From \eqref{0007} we the have that  $e^{\delta/2}G_{\delta/2} C_{e^{-\delta}}v\cle C_{e^{-\delta/2}}H^-_{\delta/2} v$ and since $e^{\delta}G_{\delta}=e^{\delta/2}G_{\delta/2}e^{\delta/2}G_{\delta/2}$ and $e^{\delta/2}G_{\delta/2}$ preserves the  order we get \eqref{0011}.

 
 \noindent {\it {Proof of \eqref{0012} }}. We have to prove that
 $C_{1}C_{e^{\delta/2}}e^{\delta/2}G_{\delta/2}w
 \cge 
 C_{1}e^{\delta/2}G_{\delta/2}C_{1}w$
 for $w\defi e^{\delta/2}G_{\delta/2}v$. From  (c) of Proposition \ref{0hh}
 this follows if we prove
 \begin{align}
 C_{e^{\delta/2}}e^{\delta/2}G_{\delta/2}w
 \cge 
 e^{\delta/2}G_{\delta/2}C_{1}w.
 \end{align}
 Denote
$ q\defi q_{e^{\delta/2}} \big(e^{\delta/2}G_{\delta/2}w\big)$. 
 If $a\le q$,
 \begin{align}
 \int_a^\infty C_{e^{\delta/2}}e^{\delta/2}G_{\delta/2}w
 =e^{\delta/2}
 =\int_{-\infty}^\infty e^{\delta/2}G_{\delta/2}C_{1}w
 \geq \int_{a}^\infty e^{\delta/2}G_{\delta/2}C_{-1}w.
 \end{align}
 For $a>q$, since $w\le C_{1}w$ point-wise from (b) of Proposition  \ref{0hh} we get
 \begin{align}
 \int_a^\infty e^{\delta/2}G_{\delta/2}w
 \geq \int_{a}^\infty e^{\delta/2}G_{\delta/2}C_{1}w;
 \end{align}

 We now prove \eqref{0eq4}.
 Let $k\defi t/\delta$ and define $u_k\defi e^\delta G_\delta \pare{H^+_\delta}^{k-1}u$ and $v_k\defi S_{t}^{\delta,-}u$.
 Using that $\norm{u_k}_1=e^\delta$ and assuming  $\delta$ small enough we get
 \begin{align}
 \nn
 \norm{S_{t}^{\delta,+}u-S_{t}^{\delta,-}u}_1 &=\norm{C_{1}u_k-v_k}_1
 \leq \norm{C_{1}u_k-u_k}_1+\norm{u_k-v_k}_1
 \\&= \pare{e^\delta-1}+\norm{u_k-v_k}_1\leq  2\delta+\norm{u_k-v_k}_1
 \label{0005}
 \end{align}

 By items (d) and (e) of Proposition \ref{0hh},
 \begin{align}
 \norm{u_k-v_k}_1
 &\nn\leq e^\delta\norm{C_{1}u_{k-1}-C_{e^{-\delta}}v_{k-1}}_1
 \\
 &\nn\le e^\delta\norm{C_{1}u_{k-1}-C_{e^{-\delta}}u_{k-1}}_1
 +e^\delta\norm{C_{e^{-\delta}}u_{k-1}-C_{e^{-\delta}}v_{k-1}}_1
 \\ \label{llkk}
 &\nn\le e^\delta\corch{\pare{e^{\delta}-1}-\pare{1-e^{-\delta}}}
 +e^\delta\norm{u_{k-1}-v_{k-1}}_1
 \\
 &\nn\le 3\delta^2+e^\delta\norm{u_{k-1}-v_{k-1}}_1.
 \end{align}
 Iterating and using that $\norm{u_1-v_1}_1\leq e^\delta\norm{u-C_{e^{-\delta}}u}_1=e^\delta\pare{1-e^{-\delta}}\le 2\delta$ we get
 \begin{align}
\nn \norm{v_k-u_k}_1&\le   3\delta^2 \sum_{j=0}^{k-2} e^{\delta j}+\norm{u_1-v_1}_1e^{\delta(k-1)}
 \\
 &\nn\le 3\delta^2 \frac {e^{\delta (k-1)}-1}{e^{\delta }-1}+2\delta e^t \le 3\delta \pare{e^t-1}+2\delta e^t
=c_1\delta.
 \end{align}
 We conclude by replacing in \eqref{0005}. \end{proof}

In order to prove Theorem \ref{thm5} we fix $T>0$, $u\in L^1(\mathbb R,\mathbb R_+)\cap L^\infty(\mathbb R,\mathbb R_+)$ and $t_0>0$. Call 
\begin{align}
  \mathcal T_{n}:=\{ k2^{-n}, k\in\mathbb N\}
\end{align}
and define the function $\rho_n:\mathbb R\times [t_0,T)\to\mathbb R_+$ as 
	\begin{equation}
	\nn
\rho_{n}(r,t):=S_t^{2^{-n},-}u,\qquad r\in\mathbb R, \quad t\in[t_0,T]\cap \mathcal T_{n}
	\end{equation}
and then define $\rho_{n}(r,t)$ for all $t\in [t_0,T]$ by linear interpolation.


%
%
Observe that $\rho_{n}$ are uniformly bounded since $\norm{u}_\infty e^T$ is a uniform bound for all $t$ and $n$. We will apply
Ascoli-Arzel\'a Theorem 
thus we need to prove equi-continuity: we will prove  space and time equi-continuity separately in Lemma \ref{0lemm1} and Lemma \ref{0lemm2} below.

We will use the following Lemma.
\begin{lemma}\label{a13} Given $\delta>0$ let $t,s\in \delta\mathbb N $ with $s<t$. For $u\in L^1(\mathbb R,\mathbb R_+)\cap L^\infty(\mathbb R,\mathbb R_+)$ let
	\begin{equation}
	\label{5.21a}
	v_{s,t}^{\delta}\defi S_t^{\delta,-}u-e^{t-s}G_{t-s}S_s^{\delta,-}u
	\end{equation}
Then 
\begin{align}\label{0eqq3}
\norm{v_{s,t}^\delta}_\infty\le \frac{2e^T\sqrt{t-s}}{\sqrt{2\pi}}.
\end{align}

\end{lemma}
\begin{proof} We write
\begin{align}
\label{a5.22}
S_t^{\delta,-}u=e^{\delta}G_\delta S_{t-\delta}^{\delta,-}u-e^{\delta}G_\delta w_{t-\delta}^\delta,\qquad w_{t'}^\delta\defi S^{\delta,-}_{t'}u-C_{e^{-\delta}}S^{\delta,-}_{t'}u
\end{align}
 Call $m$ and $h$ the integers such that $s=m\delta$ and $t=h\delta$ and iterate \eqref{a5.22} to get
\begin{align*}
S_t^{\delta,-}u=e^{t-s}G_{t-s}S_{s}^{\delta,-}u-\sum_{k=m}^{h-1}e^{h-k}G_{\pare{h-k}\delta}w_{k\delta}^\delta.
\end{align*}
Since   $\norm{G_{t'}w}_\infty\le \norm{w}_1/\sqrt{2\pi t'}$ for any $t'>0$ and any  $w\in L^1$, using that  $\norm{w_{k\delta}^\delta}_1=1-e^{-\delta}\le \delta$ we get
\begin{eqnarray*}
\norm{v_{s,t}^\delta}_\infty &\le&
\sum_{k=m}^{h-1}\norm{e^{h-k}G_{(h-k)\delta}w_{k\delta}^\delta}_\infty
\le \sum_{k=m}^{h-1} \frac{e^{(h-k)\delta}\norm{w_{k\delta}^\delta}_1}{\sqrt{2\pi (h-k)\delta}}
\le \frac{e^T \sqrt{\delta}}{\sqrt{2\pi}}\sum_{k=m}^{h-1}\frac{1}{\sqrt{h-k}}
\\&\le& \frac{e^T \sqrt{\delta}}{\sqrt{2\pi}} 2\sqrt{h-m}=\frac{2e^T\sqrt{t-s}}{\sqrt{2\pi}}. \qedhere
\end{eqnarray*}
\end{proof}

\begin{lemma}[Space equi-continuity]\label{0lemm1}
For any $\varepsilon>0$, there exist  $n_0$ and $\zeta>0$ such that 
for any $n>n_0$, any  $t\in \corch{t_0,T}$,
and any $r,r'\in \R$ such that $\abs{r-r'}<\zeta$
	\begin{align}
	\label{a5.20}
\abs{\rho_n(r,t)-\rho_n(r',t)}<\varepsilon,\qquad \forall r,r'\in \R \text{ such that  }\abs{r-r'}<\zeta
	\end{align}

\end{lemma}

\begin{proof}
Fix $\varepsilon>0$. 
Choose $n_0$  so that $\delta_0:=2^{-n_0}<t_0$ and such that $\{4e^T\sqrt{\delta_0}\}/\sqrt{2\pi}<\varepsilon/2$
Choose also $\zeta>0$ such that $\zeta \{c e^T\norm{u}_\infty\}/\sqrt{\delta_0}<\varepsilon/2$
where $c>0$ is the constant of item (f) of Proposition \ref{0hh}. We take 
$\delta=2^{-n}$ with $n>n_0$ and suppose  that $t=h\delta\in [t_0,T]$ observing that it is enough to prove \eqref{a5.20} for $t$ of this form since $\rho_n$ is defined by linear interpolation. 
Let $s\defi t-\delta_0$ and  $v_{s,t}^{\delta}\defi S_t^{\delta,-}u-e^{t-s}G_{t-s}S_s^{\delta,-}u$, then
\begin{align*}
\abs{S_t^{\delta,-}u\pare{r}-S_t^{\delta,-}u\pare{r'}}=
\abs{e^{t-s}G_{t-s}S_s^{\delta,-}u(r)-e^{t-s}G_{t-s}S_s^{\delta,-}u(r')}
+
\abs{v_{s,t}^\delta(r)-v_{s,t}^\delta(r')}.
\end{align*}
By item (f) of Proposition \ref{0hh} and the bound $\norm{S_s^{\delta,-}u}_{\infty}\le e^s \norm{u}_\infty$,  for $r,r'\in\R$ as in \eqref{a5.20} and by the choice of $\zeta$ we have
\begin{align*}
\abs{e^{t-s}G_{t-s}S_s^{\delta,-}u(r)-e^{t-s}G_{t-s}S_s^{\delta,-}u(r')}\le
\frac{c e^T\abs{r-r'} \, \norm{u}_\infty}{\sqrt{\delta_0}}<\varepsilon/2.
\end{align*}
By the choice of $\delta_0$  and from   \eqref{0eqq3} we get
$\abs{v_{s,t}^\delta(r)-v_{s,t}^\delta(r')}\le 2\norm{v_{s,t}^\delta}_\infty<\varepsilon/2$ 
which concludes the proof. \end{proof}

\begin{lemma}[Time equi-continuity]\label{0lemm2}
For any $\varepsilon>0$, there exist  $n_0$ and $\zeta>0$ such that 
for any $n>n_0$ and any $r\in \R$,
\begin{align}
\label{a5.25}
\abs{\rho_n(r,t)-\rho_n(r,t')}<\varepsilon,\qquad \forall t,t'\in \corch{t_0,T}\text{ such that } \abs{t-t'}<\zeta
\end{align}
\end{lemma}

\begin{proof}
Fix $\varepsilon>0$ and
let $\zeta'$ and $n_0$ be the parameters given by Lemma \ref{0lemm1} associated to $\varepsilon'\defi \varepsilon/4$.
Take  $\zeta$ such that
$\dis{\frac{2e^T}{\sqrt{2\pi}}\sqrt{\zeta}<\varepsilon/4}$,
$\dis{\frac{2\sqrt{\zeta} e^T}{\sqrt{2\pi}\zeta'}e^{-\frac{\pare{\zeta'}^2}{2\zeta}}<\varepsilon/4}$ and $\pare{e^\zeta-1}e^T\norm{u}_\infty<\varepsilon/4$.
For $r\in\R$, $\delta:=2^{-n}$,  such that $n>n_0$ we first consider  $t,t'\in\corch{t_0,T}\cap\delta\N$ such that $t<t'<t+\zeta$
and we have to prove that
\begin{align}\label{00rr}
\abs{S^{\delta,-}_{t'}u(r)-S^{\delta,-}_{t}u(r)}=\abs {v_{t,t'}+e^{t'-t}G_{t'-t}S^{\delta,-}_tu(r)-S^{\delta,-}_tu(r)}<\varepsilon.
\end{align}
Using \eqref{0eqq3} we get 
\begin{eqnarray}\nn
&&
\hskip-2cm\abs{S^{\delta,-}_{t'}u(r)-S^{\delta,-}_{t}u(r)}\le \frac{\varepsilon}4+\abs{{t'-t}G_{t'-t}S^{\delta,-}_tu(r)-S^{\delta,-}_tu(r)}
\\&&\le \frac{\varepsilon}4+ \abs{{t'-t}G_{t'-t}S^{\delta,-}_tu(r)-G_{t'-t}S^{\delta,-}_tu(r)}+
\abs{G_{t'-t}S^{\delta,-}_tu(r)-S^{\delta,-}_tu(r)}
\label{0dd}\end{eqnarray}
We have
\begin{align}
\label{5.27}
\abs{{t'-t}G_{t'-t}S^{\delta,-}_tu(r)-G_{t'-t}S^{\delta,-}_tu(r)}\le (e^{t'-t}-1)\norm{G_{t'-t}S_t^{\delta,-}u}_\infty\leq 
(e^{t'-t}-1)e^T\norm{u}_\infty<\varepsilon/4,
\end{align}
We next estimate
\begin{eqnarray}\nn
\int_{-\infty}^\infty G_{t'-t}\pare{r,r'}\abs{S_t^{\delta,-}u(r')-S_t^{\delta,-}u(r)}dr'
&\le&
\frac{\varepsilon}4+e^T\int_{\abs{r'-r}\ge\zeta'} G_{t'-t}\pare{r,r'}dr'
\\&\le&  \frac{\varepsilon}4+ \le
\frac{2\sqrt{\zeta}}{\sqrt{2\pi}\zeta'}e^{-\frac{\pare{\zeta'}^2}{2\zeta}}
\label{5.28}
\end{eqnarray}
Inserting the estimates\eqref{5.27} and \eqref{5.28} in \eqref{0dd} we get \eqref{00rr}.

For generic  $t,t'\in [t_0,T]$ such that
$t<t'<t+\zeta$ we consider $\delta=2^{-n}$ as before and we consider $t^-,t^+\in\delta \N$ such that
$t^-\le t<t^-+\delta$ and $t^+-\delta< t\le t^+$. Then 
\begin{align*}
\abs{\rho_n\pare{r,t'}-\rho_n\pare{r,t}}
\le
\max_{t_1,t_2\in [t^-,t^+]\cap\delta\N}\abs{S_{t_2}^{\delta,-}u(r)-S_{t_1}^{\delta,-}u(r)}
\end{align*}
%
Then from \eqref{00rr} we get \eqref{a5.25}.\end{proof}

{\bf{Proof of Theorem \ref{thm5}}}. For any function $w\in L^1$ we define
	$$F(r;w):=\int_r^\infty w(r')dr'$$ 
From Ascoli-Arzel\'a Theorem we have convergence by subsequences of $(\rho_n)_{ n\ge 1}$. Let $\psi$ be any limit point of $\rho_n)$.  Observe that for each $t\in [t_0,T]\cap  \mathcal T_n$ we have that $\rho_n=S^{2^{-n},-}_tu\in L^1$. Since by Proposition \ref{thm031}  $F(r; S^{2^{-n},-}_tu)$ is a non increasing function of $n$ it converges as $n\to\infty$. Then by dominate convergence we have that 
for any $r\in \mathbb R$ and $t\in [t_0,T]\cap  \mathcal T_n$
	\begin{equation}
	\label{corr8.1.12}
\lim_{n\to \infty} F(r; S^{2^{-n},-}_tu) = F(r;\psi(\cdot,t))
 	\end{equation}
Thus all limit functions
$\psi(r,t)$ agree on $t\in [t_0,T]\cap  \mathcal T_n$
and since they are continuous they agree on the whole $[t_0,T]$,
thus the sequence $\rho_n(r,t)$ converges
in sup-norm  as $n\to \infty$ to a continuous function
$\psi(r,t)$ (and not only by subsequences).  Observe that from \eqref{corr8.1.12} we also have
$F(r;S^{2^{-n},-}_tu) \le F(r; \psi(\cdot,t))$ for any $n$ and $t \in  \mathcal T_{n}$. \qed

\section{Proof of Theorem \ref{hydrodynamics}}
\label{s6}

Fix $t > 0$, choose $\delta\in\{2^{-n}t, n \in \mathbb N\}$  and $k$ such that $k\delta=t$. Take $\ux_0$ as in Theorem \ref{hydrodynamics}, that is, iid continuous random variables with density $\rho$.   
By Proposition \ref{domination}, there is a coupling between the barriers and $N$-BBM such that, for increasing $\varphi$,
\begin{align}\label{6.1}
\hat \pi^{N,\delta,-}_t \varphi \,\le\,  \hat \pi^N_t \varphi \,\le\,  \hat \pi^{N,\delta,+}_t \varphi  
\end{align}
where $\hat \pi$ are the empiric measures associated to the coupled processes $\hux$ of Proposition \ref{domination} with initial condition $\ux_0$ in the three coordinates. In Theorem \ref{thm4} we have proven that under this initial conditions,  $\pi^{N,\delta,\pm}_t \varphi$ converge to $\int \varphi\, S^{\delta,\pm}_t\rho$ almost surely and in $L^1$. This limit was proven using only the generic LLN of Proposition \ref{propo12} which only uses the mean and variance of  $g^{\pm}_\varphi$ associated to $\hat\pi^{N,\delta,\pm}_t \varphi$ in \eqref{gbt+} and \eqref{gbt-} as functions of $\cB$ which coincide with the corresponding to  $\hat\pi^{N,\delta,\pm}_t \varphi$ as functions of $\cB^\pm$.  We can conclude that the same convergence holds for the hat-variables.

On the other hand, by \eqref{6.1}, 
\begin{align}
  \big| \pi^{N}_t \varphi -\hat \pi^{N,\delta,\pm}_t \varphi \big| &\le  \big|\pi^{N,\delta,+}_t \varphi - \hat \pi^{N,\delta,-}_t \varphi\big| \le  \|\varphi\|\,c\delta,
\end{align}
by \eqref{0eq4}. We can conclude using Theorem \ref{thm4} that 
\begin{align}
\lim_{N\to\infty}  \big| \pi^{N}_t \varphi - \hbox{$\int \varphi\, S^{\delta,\pm}_t\rho$}\big| &\le  \|\varphi\|\,c\delta, 
\quad\hbox{a.s.\/  and in }L^1. \label{4.199}
\end{align}
Taking $\delta\to0$ along dyadics, we get a function $\int \varphi\,\psi: = \lim_{\delta\to0}\int \varphi\,S^{\delta,\pm}_t\rho $  in $L^1$ and
\begin{align}
  \lim_{N\to\infty}  \big| \pi^{N}_t \varphi - \hbox{$\int \varphi\, \psi$}\big|\quad\hbox{a.s.\/  and in }L^1.\qquad\qed
\end{align}

\section{Proof of Theorem \ref{u-existence}}
			\label{s7}

Fix a density $\rho$ and assume there is a continuous curve $L=(L_t:t\ge 0)$ and density functions $u= (u(r,t):r\in \R, t\ge 0)$ such that $(u,L)$ solves the free boundary problem. It is convenient to stress the semigroup property of the solution so we call the solution $S_t\rho :=u(\cdot,t)$ and notice that the operator $S_t$ is a semigroup.
The following theorem shows that the solution is in between the barriers.
  	\begin{theorem}
	\label{teo18}
Let $(u,L)$ be a solution of the FBP in $(0,T]$. Let $t\in (0,T]$ and $\delta \in \{2^{-n}t: n \in \mathbb N\}$.  Then
 \begin{equation}
	\label{77}
S^{\delta,-}_{t} \ro  \preccurlyeq   S_t\rho \preccurlyeq
S^{\delta,+}_{t}   \ro ,\qquad t=k\delta.
  \end{equation}
	\end{theorem}
We show \eqref{77} first for time $\delta = 2^{-n}t$ and then use induction to extend to times $k\delta$.

			\begin{proposition}
			\label{propM2.1}
For all $r\in\R$ we have
 \begin{align}
		\label{m.2.6}
	F(r;S_t\rho)&\le F(r;S^{\delta,+}_\delta\ro),\\
	\label{m.2.11}
	F(r;S_t\rho)&\ge F(r;S^{\delta,-}_\delta\ro),
\end{align}
\end{proposition}
                        \begin{proof}
If $r< L^{\delta,+}_\delta$, by definition of $L^{\delta,+}_\delta$ we have
$F(r;S^{\delta,+}_t\ro)= F(L^{\delta,+}_t;\ro)= 1\ge
F(r;S_t\rho)$. 
If $r>L^{\delta,+}_\delta$, we have
 		\begin{align}
			 \nn
F(r;S^{\delta,+}_t\ro)&= e^\delta  \int \ro(x) P_{x}\big(B_\delta \ge r\big)dx 
          \ge e^\delta  \int \ro(x) P_{x}\big(B_\delta \ge r;\tau^L >\delta\big)dx =F(r;S_t\rho),
			\label{m.2.8}
			\end{align}
using the Brownian motion representation \eqref{urtbm} of $S_t\rho$ with $\tau^L=\inf\{t>0: B_t\le L_t\}$. This shows \eqref{m.2.6}.
                      
To show \eqref{m.2.11} recall the cut operator \eqref{cut1} and denote $\rho_0 := C_{e^{-\delta}}\ro$ and $\rho_1:=\ro-\rho_0$. We then have $\int \rho_0(r)dr=e^{-\delta}$ and 
 \begin{equation}
		\label{m.2.10}
	F(r,S^{\delta,-}_\delta\ro)=  e^\delta  \int \rho_0(x) P_{x}\big(B_\delta \ge r\big)dx.
\end{equation}
  We have
	\begin{align}
			 \nn
F(r,S_t\rho)&= e^\delta  \int \ro(x) P_{x}\big(B_\delta \ge r;\tau^L >\delta\big)dx
\\ &=
 e^\delta  \int \ro_0(x) P_{x}\big(B_\delta \ge r\big)dx -
  e^\delta  \int \ro_0(x) P_{x}\big(B_\delta \ge r;\tau^L \le \delta\big)dx \nn \\
  & \qquad +\, e^\delta  \int \ro_1(x) P_{x}\big(B_\delta\ge r;\tau^L >\delta\big)dx.\nn
			\label{m.2.12}
			\end{align}
Thus, recalling \eqref{m.2.10}, it suffices to show
	 \begin{equation}
		\label{m.2.13}
e^\delta \int \ro_0(x) P_{x}\big(B_\delta \ge r;\tau^L \le \delta\big)dx\le 
 e^\delta  \int \ro_1(x) P_{x}\big(B_\delta \ge r;\tau^L >\delta\big)dx,
 \end{equation}
We have that
\begin{equation}
		\label{m.2.15}
  e^\delta
  \int \ro_0(x) P_x(\tau^L\le \delta)dx= e^\delta
  \int \ro_1(x) P_x(\tau^L> \delta)dx,
\end{equation}
where the last identity follows from subtracting the following identities 
\begin{align}
  e^\delta \int (\ro_0(x)+\ro_1(x)) P_x(\tau^L>\delta)dx&=\int S_\delta\ro(x)dx=1\nn \\
e^\delta \int \ro_0(x) dx&=e^\delta\int C_{e^{-\delta}}\ro(x)  dx=1\nn
\end{align}
We rewrite \eqref{m.2.13} as
   \begin{equation}
   \label{aaaa}
   \int \ro_0(x) \int_0^\delta h^L_{x}(ds) P_{L_s,s}\big(B_\delta \ge r\big)dx\le 
   \int \ro_1(x) P_x(\tau^L\ge \delta) P_{x}\big(B_\delta \ge r|\tau^L >\delta\big)dx			
\end{equation}
where $P_{y,s}$ denotes the law of a  Brownian motion starting from $y$ at time $s$ and $h^L_{x}$ denotes the cumulative distribution function of $\tau^L$ under $P_{x,0}$. In Section 10.3.2 of \cite{MR3497333} it has been proved that if $L$ is a continuous curve then for all $r$
\begin{equation*}
P_{L_{t};t}\big(B_\delta \ge r\big)\le P_x\big(B_\delta \ge r\,\big|\, {\tau^L} >\delta\big),\quad x > L_0,\; t\in [0,\delta)
	\end{equation*}
From this and from \eqref{m.2.15} inequality \eqref{aaaa} easily follows. 
\end{proof}

\emph{Remark. } Dividing \eqref{m.2.13} by \eqref{m.2.15}, we have proven the inequality
\begin{align}
\label{339}
 P_{\ro_1}(B_\delta\ge r|\tau^L>\delta)- P_{\ro_0}(B_\delta\ge r|\tau^L\le\delta) \ge 0.
\end{align}
where $P_{\ro_i}$ is the law of Brownian motion with initial distribution \[P_{\ro_i}(B_0\in A) =\|\ro_i\|_1^{-1}\int_A \ro_i(x)dx.\]



\begin{proof}
  [Proof of Theorem \ref{teo18}]
   \label{secM.33}
Recalling the definitions of $C_m$ and $G_t$, 
  Proposition \ref{propM2.1} shows the following inequalities for $n=1$:
  \begin{align}
    \label{uu13}
    \big(e^\delta G_\delta C_{e^{-\delta}}\big)^n\ro \preccurlyeq S_{n\delta}\rho\preccurlyeq \big(C_{1}e^\delta G_\delta\big)^n\ro
   \end{align}
Apply \eqref{uu13} with $n=1$ to $S_{n\delta}\rho$ to get 
\begin{align}
          \label{uu15}
     (e^\delta G_\delta C_{1-e^{-\delta}})S_{n\delta}\rho \preccurlyeq S_\delta S_{n\delta}\rho\preccurlyeq (C_{e^\delta-1}e^\delta G_\delta)S_{n\delta}\rho
\end{align}
Apply each inequality in \eqref{uu13} to the corresponding side in \eqref{uu15} to obtain
\begin{align}
    (e^\delta G_\delta C_{e^{1-\delta}})^{n+1}\ro \preccurlyeq   (e^\delta G_\delta C_{e^{-\delta}})S_{n\delta}\ro \preccurlyeq  S_{(n+1)\delta}\ro\preccurlyeq (C_{1}e^\delta G_\delta)S_{n\delta}\ro\preccurlyeq (C_{1}e^\delta G_\delta)^{n+1}\ro\nn
\end{align}
where we have used that both $G_\delta$ and $C_m$ are monotone, by Proposition \ref{0hh}. 
\end{proof}

\section{Traveling waves}
\label{s8}

\paragraph{Traveling waves} Fix $N$ and let $\ux_t$ be $N$-BBM. Let  $\ux'_t := \{x-\min \ux_t:x\in\ux_t\}$ be the process as seen from the leftmost particle. In this process there is always a particle at the origin. The following theorem has been proven by Durrett and Remenik \cite{MR2932664} for a related Brunet-Derrida process. The proof in this case is very similar so we skipt it. 
\begin{theorem}
  \label{micro-tw}
 $N$-{\rm BBM} as seen from the leftmost particle is Harris recurrent. Denote  $\nu_N$ its unique invariant measure. Under $\nu_N$ the process has an asymptotic speed $\alpha_N$ given by
\begin{align}
  \alpha_N = (N-1)\, \nu_N\big[\min(\ux \setminus \{0\})\big],
\end{align}
that is the rate of branching of the $N-1$ rightmost particles times the expected distance between the leftmost particle an the second leftmost particle. 

$N$-{\rm BBM} starting with an arbitrary configuration converges in distribution to $\nu_N$ and 
\begin{align}
  \lim_{t\to\infty} \frac{\min\ux_t}{t} = \alpha_N.
\end{align}
Furthermore, $\alpha_N$ converges to the asymptotic speed of the first particle in {\rm BBM} with a finite initial configuration:
\begin{align}
\label{ans2}
  \lim_{N\to\infty} \alpha_N = \sqrt 2.
 \end{align}
\end{theorem}
The analogous to limit \eqref{ans2} was proven by Berard and Gou\'er\'e \cite{MR2669438} and Durrett and Mayberry \cite{MR2639750} for Brunet-Derrida systems.  

The traveling wave solutions of the free boundary problem \eqref{u2}\eqref{u3}\eqref{u4} are of the form $u(r,t)=w(r-\alpha t)$, where $w$ must satisfy
\begin{align}
\label{macro-tw}
  \frac12 w'' + \alpha w' + w =0, \quad w(0)=0, \quad \int_0^\infty w(r)dr =1. 
\end{align}
Groisman and Jonckheere \cite{gj1,gj2} observed that 
for each speed $\alpha\ge \alpha_c = \sqrt2$ there is a solution $w_\alpha$ given by
\begin{align}
w_\alpha(x) &=
           \begin{cases}
             M_\alpha\, x e^{-\alpha x}&\hbox { if }\alpha = \sqrt 2\\
M_\alpha\, e^{-\alpha x}\sinh \big(x\sqrt{\alpha^2-2}  \big)&\hbox { if }\alpha > \sqrt 2\\
           \end{cases}
\end{align}
where $M_\alpha$ is a normalization constant such that $\int w_\alpha=1$. In fact $w_\alpha$ is the unique quasi stationary distribution for Brownian motion with drift $-\alpha$ and absorption rate $w'(0)=1$; see Proposition 1 of Mart\'\i nez and San Mart\'\i n \cite{MR1303922}. More precisely, calling $\cL_\alpha w= \frac12 w'' +\alpha w'$, we have that $w_\alpha$ is the unique eigenvector for $\cL_\alpha$ with eigenvalue $-1$. See \cite{gj1} for the relation between quasi stationary distributions for absorbed Brownian motion and traveling wave solutions for the free boundary problem.

Let $\ux_t$ be the $N$-BBM process with  initial configuration sampled from the stationary measure $\nu^N$. Show that the empirical distribution of $\ux_t$ converges to a measure with density $w_{\sqrt 2}(\cdot-t\sqrt 2)$, as $N\to\infty$. This would be a \emph{strong selection principle} for $N$-BBM \cite{gj1,MR3568046}; the weak selection principle is already contained in \eqref{ans2}, the stationary speed for the finite system converges to the minimal speed in the macroscopic system. A way to show this limit would be to control the particle-particle correlations in the $\nu_N$ distributed initial configuration. If instead we start with independent particles with distribution $w_{\sqrt 2}$, then we can use Theorem 1 and the fact that $w_{\sqrt 2}(r-t\sqrt 2)$ is a strong solution of the free boundary problem to prove converge of the empirical measure to this solution.


\section*{Acknowledgments}  PAF thanks Pablo Groisman, Matthieu Jonckheere and Julio Rossi for illuminating discussions on qsd and free boundary problems and  existence of solutions for the pde. PAF thanks Gran Sasso Science Institute and University of Paris Diderot for warm hospitality. We thank kind hospitality at Institut Henry Poincar\'e, during the trimester \emph{Stochastic dynamics out of equilibrium} where part of this work was performed. 

\bibliographystyle{alpha}

\bibliography{nbbm}

\begin{thebibliography}{CDMGP14b}

\bibitem[BBD17]{bbd}
Julien {Berestycki}, Eric Brunet, and Bernard Derrida.
\newblock {Exact solution and precise asymptotics of a Fisher-KPP type front}.
\newblock {\em arXiv:1705.08416v1}, May 2017.

\bibitem[BD97]{brunet-derrida}
Eric Brunet and Bernard Derrida.
\newblock Shift in the velocity of a front due to a cutoff.
\newblock {\em Phys. Rev. E}, 56:2597--2604, Sep 1997.

\bibitem[BG10]{MR2669438}
Jean B\'erard and Jean-Baptiste Gou\'er\'e.
\newblock Brunet-{D}errida behavior of branching-selection particle systems on
  the line.
\newblock {\em Comm. Math. Phys.}, 298(2):323--342, 2010.

\bibitem[CDMGP14a]{MR3282863}
Gioia Carinci, Anna De~Masi, Cristian Giardin{\`a}, and Errico Presutti.
\newblock Hydrodynamic limit in a particle system with topological
  interactions.
\newblock {\em Arab. J. Math. (Springer)}, 3(4):381--417, 2014.

\bibitem[CDMGP14b]{MR3198665}
Gioia Carinci, Anna De~Masi, Cristian Giardin{\`a}, and Errico Presutti.
\newblock Super-hydrodynamic limit in interacting particle systems.
\newblock {\em J. Stat. Phys.}, 155(5):867--887, 2014.

\bibitem[CDMGP16]{MR3497333}
Gioia Carinci, Anna De~Masi, Cristian Giardin\`a, and Errico Presutti.
\newblock {\em Free boundary problems in {PDE}s and particle systems},
  volume~12 of {\em SpringerBriefs in Mathematical Physics}.
\newblock Springer, [Cham], 2016.

\bibitem[DM10]{MR2639750}
Rick Durrett and John Mayberry.
\newblock Evolution in predator-prey systems.
\newblock {\em Stochastic Process. Appl.}, 120(7):1364--1392, 2010.

\bibitem[DMF15]{MR3336872}
Anna De~Masi and Pablo~A. Ferrari.
\newblock Separation versus diffusion in a two species system.
\newblock {\em Braz. J. Probab. Stat.}, 29(2):387--412, 2015.

\bibitem[DMFP15]{MR3304749}
Anna De~Masi, Pablo~A. Ferrari, and Errico Presutti.
\newblock Symmetric simple exclusion process with free boundaries.
\newblock {\em Probab. Theory Related Fields}, 161(1-2):155--193, 2015.

\bibitem[DR11]{MR2932664}
Rick Durrett and Daniel Remenik.
\newblock Brunet-{D}errida particle systems, free boundary problems and
  {W}iener-{H}opf equations.
\newblock {\em Ann. Probab.}, 39(6):2043--2078, 2011.

\bibitem[GJ13]{gj1}
Pablo {Groisman} and Matthieu Jonckheere.
\newblock {Front propagation and quasi-stationary distributions: the same
  selection principle?}
\newblock {\em arXiv:1304.4847}, April 2013.

\bibitem[GJ16]{gj2}
Pablo {Groisman} and Matthieu Jonckheere.
\newblock {Front propagation and quasi-stationary distributions for
  one-dimensional L\'evy processes}.
\newblock {\em arXiv:1609.09338}, September 2016.

\bibitem[Lee17]{JLee}
J.~Lee.
\newblock {Existence of solutions for a free boundary problem}.
\newblock {\em In preparation}, 2017.

\bibitem[Mai13]{MR3088376}
Pascal Maillard.
\newblock The number of absorbed individuals in branching {B}rownian motion
  with a barrier.
\newblock {\em Ann. Inst. Henri Poincar\'e Probab. Stat.}, 49(2):428--455,
  2013.

\bibitem[Mai16]{MR3568046}
Pascal Maillard.
\newblock Speed and fluctuations of {$N$}-particle branching {B}rownian motion
  with spatial selection.
\newblock {\em Probab. Theory Related Fields}, 166(3-4):1061--1173, 2016.

\bibitem[MSM94]{MR1303922}
Servet Mart\'{\i}nez and Jaime San~Mart\'{\i}n.
\newblock Quasi-stationary distributions for a {B}rownian motion with drift and
  associated limit laws.
\newblock {\em J. Appl. Probab.}, 31(4):911--920, 1994.

\end{thebibliography}

\def\at{,\ }
\def\email#1{{\tt #1}}
{
Anna De Masi\at
              Universit\`a di L'Aquila, 67100 L'Aquila, Italy \\
                            \email{demasi@univaq.it}           
         
  Pablo A. Ferrari\at
              Universidad de Buenos Aires, DM-FCEN, 1428 Buenos Aires, Argentina\\
              \email{pferrari@dm.uba.ar}
 
          Errico Presutti\at
              Gran Sasso Science Institute, 67100 L'Aquila, Italy\\
              \email{errico.presutti@gmail.com} 

           Nahuel Soprano-Loto\at
              Gran Sasso Science Institute, 67100 L'Aquila, Italy\\
              \email{sopranoloto@gmail.com}
}

\end{document}